\documentclass[reqno]{article} 

\usepackage{amsmath,amssymb,amsfonts,amsthm}
\usepackage{enumitem}
\usepackage{fullpage}
\usepackage[usenames,dvipsnames]{xcolor}
\usepackage[colorlinks=true, pdfstartview=FitV,linkcolor=ForestGreen,citecolor=ForestGreen, urlcolor=blue]{hyperref}

\newtheorem{theo}{\bf Theorem}[section]
\newtheorem{lem}[theo]{\bf Lemma}
\newtheorem{pro}[theo]{\bf Proposition}
\newtheorem{cor}[theo]{\bf Corollary}
\newtheorem{defi}[theo]{\bf Definition}
\theoremstyle{remark}
\newtheorem{rem}[theo]{Remark}

\newcommand{\R}{{\mathbb R}}
\def\a{\alpha}
\def\e{\epsilon}
\def\d{\delta}

\numberwithin{equation}{section}

\begin{document}

\title{\bf The twin blow-up method for Hamilton-Jacobi equations in higher dimension
}
\author{Nicolas Forcadel, Cyril Imbert and R\'egis Monneau}

\maketitle

\begin{abstract}
In this paper, we show how to extend the twin blow-up method recently developped by the authors (\textit{Comptes Rendus. Math.}, 2024), in order to obtain a new comparison principle for an evolution coercive Hamilton-Jacobi equation posed in a domain of an Euclidian space of any dimension and supplemented with a boundary condition. The  method allows dealing with the case where tangential variables and the  variable corresponding to the normal gradient of the solution  are strongly coupled at the boundary. We elaborate on a method introduced by P.-L. Lions and P. Souganidis (\textit{Atti Accad. Naz. Lincei}, 2017). Their argument relies on a {single} blow-up procedure after rescaling the semi-solutions to be compared while two simultaneous blow-ups are performed in this work,  one for each  variable of the classical doubling variable technique.  A  one-sided Lipschitz estimate satisfied by a combination of the two blow-up limits plays a key role. 
\end{abstract}

\tableofcontents

\section{Introduction}

This work is concerned with strong uniqueness (comparison principle) of viscosity solutions to a Hamilton-Jacobi equation of evolution type of the form,
\begin{equation}\label{eq::b1}
u_t+ H(X,Du)=0 \quad \mbox{on}\quad (0,T)\times \Omega
\end{equation}
where $X:=(t,x)$, supplemented with the (desired) boundary condition
$$u_t+ F(X,Du)=0 \quad \mbox{on}\quad (0,T)\times \partial \Omega$$
and the initial condition
\begin{equation}\label{eq::b2}
u(0,\cdot)=u_0\quad \mbox{on}\quad \left\{0\right\}\times \overline \Omega.
\end{equation}
The spatial domain $\Omega$ is a subset of the Euclidian space of dimension $d \ge 1$. We will first see how to deal with a half-space  and we will then consider the
case of a $C^1$ bounded domain. 

Since the desired boundary condition can be lost when characteristics reach $\partial \Omega$, it has to be imposed in a weak sense.
In the viscosity solution framework, the weak sense means that either the desired boundary condition is satisfied or the PDE is satisfied on the boundary.
More precisely, subsolutions and supersolutions of \eqref{eq::b1} are assumed to satisfy at the boundary the following inequalities,
\begin{equation}\label{eq::b3}
\left\{\begin{array}{lll}
u_t+ \min(F,H)(X,Du)\le 0 &\quad \mbox{on}\quad (0,T)\times \partial \Omega &\quad \mbox{(subsolutions)},\\
u_t+ \max(F,H)(X,Du)\ge 0 &\quad \mbox{on}\quad (0,T)\times \partial \Omega &\quad \mbox{(supersolutions)}.\\
\end{array}\right.
\end{equation}

We present in the introduction the comparison principle for \eqref{eq::b1}, \eqref{eq::b2}, \eqref{eq::b3} with $\Omega = \R^{d-1} \times (0,+\infty)$. 
In order to present the structure conditions imposed to the Hamiltonian $H$ and the nonlinearity $F$ associated with the boundary condition, 
we set $x=(x',x_d) \in \R^{d-1}\times [0,+\infty)$ and $p=(p',p_d)\in \R^{d-1}\times \R$ for the variable for the gradient. In particular,  $p'$ corresponds
to the  space tangential gradient of the solutions and $p_d$ to the normal gradient. In the following assumption,  $\omega,\omega_L$ denote moduli of continuity.
\begin{equation}\label{eq::b4}
\left\{\begin{array}{l}
\mbox{\bf i) (Continuity and bound)}\\
         \mbox{$H:[0,T]\times \overline \Omega \times \R^d \to \R$ is continuous}\\
         \mbox{the map $X\mapsto H(X,0)$ is bounded}.\\
\\
\mbox{\bf ii) (Uniform continuity in the gradient)}\\ 
\mbox{For any $L>0$, we have for all $X\in [0,T]\times \overline \Omega$ and $p,q\in [-L,L]^d$}\\[2mm]
|H(X,p)-H(X,q)| \le \omega_L(|p-q|).\\
\\
\mbox{\bf iii) (Continuity in the tangential variables)}\\
\mbox{For $X=(t,x',x_d)$ and $Y=(s,y',x_d)$ with $t,s \in [0,T]$ and $x',y' \in \R^{d-1}$ and $x_d \ge 0$}\\[2mm]
H(Y,p)-H(X,p) \le \omega(|Y-X|\left(1+|p'|+\max\left\{0,H(X,p)\right\}\right)).\\
\\
\mbox{\bf iv) (Uniform normal coercivity)}\\
\mbox{For any $L>0$, we have}\\
\displaystyle \lim_{|p_d|\to +\infty} \inf \{ H(X,p',p_d) : {X\in [0,T]\times \overline \Omega,\ p'\in  [-L,L]^{d-1}}\}= +\infty.\\
\end{array}\right.
\end{equation}
and making artificially appear the dependence on $x\in \overline \Omega$ for $F$ (in order to unify the presentation of $H$ and $F$), we consider
\begin{equation}\label{eq::b5}
\left\{\begin{array}{l}
\mbox{\bf i) (Continuity, bound and monotonicity)}\\
         \mbox{$F:[0,T]\times \partial \Omega \times \R^d \to \R$ is continuous,} \\
         \mbox{the map $X\mapsto F(X,0)$ is bounded,}\\
\mbox{the map $p_d \mapsto F(X,p',p_d)$ is nonincreasing.}\\
\\
\mbox{\bf ii) (Uniform continuity in the gradient)}\\ 
\mbox{For any $L>0$, we have  for all $X\in [0,T]\times \partial \Omega$ and $p,q\in [-L,L]^d$}\\
|F(X,p)-F(X,q)| \le \omega_L(|p-q|). \\
\\
\mbox{\bf iii) (Continuity in the tangential variables)}\\
\mbox{for all $X,Y\in [0,T]\times \partial \Omega$ and $p\in \R^d$,}\\
F(Y,p)-F(X,p) \le \omega(|Y-X|\left(1+|p'|+\max\left\{0,\max(F,H)(X,p)\right\}\right)).\\
\\
\mbox{\bf iv) (Uniform normal semi-coercivity)}\\
\mbox{For any $L>0$, we have}\\
\displaystyle \lim_{p_d\to -\infty} \inf \{ F(X,p',p_d) : {X\in [0,T]\times \partial \Omega,\ p'\in  [-L,L]^{d-1}}\} = +\infty.\\
\end{array}\right.
\end{equation}
Under the previous structural conditions, sub and super-solutions of the Hamilton-Jacobi equation under study can be compared. 
\begin{theo}[A comparison principle with strong tangential coupling] \label{th::b7}
Let $T>0$ and assume that $H,F$ satisfy \eqref{eq::b4}-\eqref{eq::b5}. Assume that the initial data $u_0$ is uniformly continuous.
Let $u,v: [0,T)\times \overline \Omega \to \R$ be two functions with $u$ upper semi-continuous and $v$ lower semi-continuous. 
Assume that $u$ (resp. $v$) is a viscosity subsolution (resp. supersolution) of \eqref{eq::b1}-\eqref{eq::b2}. Assume moreover that there exists a constant $C_T>0$ such that
\begin{equation}\label{eq::a10}
u\le u_0+C_T\quad \mbox{and}\quad v\ge u_0-C_T\quad \mbox{on}\quad [0,T)\times \overline \Omega.
\end{equation}
If we have
$$u(0,\cdot)\le u_0 \le v(0,\cdot) \quad \mbox{on}\quad \left\{0\right\}\times \overline \Omega$$
then we have
$$u\le v \quad \mbox{on}\quad [0,T)\times \overline \Omega.$$
\end{theo}
\begin{rem}
  A simplified version of Theorem \ref{th::b7} is presented in \cite{FIM2}. It was assumed in this note that dimension $d=1$ and that initial data are Lipschitz continuous.
  Some details were skept and they are presented in this new work. 
\end{rem}
\begin{rem}
  In Section \ref{s5}, we also extend this result to the case where $\Omega$ is a $C^1$ bounded open set.
\end{rem}
\begin{rem}
Notice that, given \eqref{eq::b4}, we can always define the state constraint boundary function
\[H^-(X,p',p_d):=\inf_{q_d\le p_d} H(X,p',q_d)\quad \text{for}\quad X\in [0,T]\times \partial\Omega\quad \text{and}\quad p=(p',p_d)\in \R^{d-1}\times \R\]
and it satisfies \eqref{eq::b5}. {Up to our knowledge, the comparison principle was also an open problem for  $F=H^-$ in this generality.}
\end{rem}
\begin{rem}
{Notice that in Theorem \ref{th::b7}, semi-coercivity of $F$ in condition \eqref{eq::b5} iv) can be replaced by the weak continuity of the subsolution $u$ on the boundary $(0,T)\times \partial \Omega$, using \cite[Proposition~3.12]{FIM1}) and replacing $F$ by $F_1:=\max(F,H^-)$.}
\end{rem}

\paragraph{Main contribution.}

When comparing non-Lipschitz sub/supersolutions (for instance after constructing solutions by Perron's method), a strong  coupling between tangential coordinates $(t,x') \in [0,T]\times \R^{d-1}$ and the normal gradient $\partial_{x_d} u$ is well identified in the literature as a technical difficulty, especially when this coupling arises in the boundary condition, see for instance \cite{B2,B,BC,LS1,LS2}.

It is standard to make the (strong) assumption of uniform continuity in time $t$, uniformly in the gradient $Du$. Such an assumption is not satisfied by the following simple example,
\begin{equation} \label{eq:hj-ex}
  \begin{cases}
u_t + a(X)|Du|=0 & \quad \text{in}\quad (0,T)\times  \Omega,\\
u_t  +\max\left\{0,- b(X)\partial_{x_d} u\right\}=0 &\quad \text{in}\quad (0,T)\times \partial \Omega
\end{cases}
\end{equation}
when $a,b\ge 1$ are bounded Lipschitz continuous functions {(here with $b(t,x)=b(t,x',0)$)}.

\paragraph{Comparison with known results.}

J. Guerand \cite{G} proved a comparison principle in our geometric setting in dimension $d=1$  in the case where  $H$ and $F$ {are} independent of $(t,x)$. 
She also proved a comparison principle for non-coercive Hamiltonians. 

P.-L. Lions and P. Souganidis \cite{LS2} introduced a new method for proving  comparison principles for bounded uniformly continuous sub/supersolutions for equations posed on junctions
with several branches (or half-spaces). They use a blow-up argument that reduces the study to a 1D problem. They  show the comparison principle in the case of Kirchoff-type boundary conditions
and non-convex Hamiltonians. As far as $(t,x)$ dependence is concerned, their method allows them to handle Hamiltonians that are Lipschitz continuous in $t$, see  \cite[Assumption (4)]{LS2}.

 This result is generalized by G. Barles and E. Chasseigne \cite[Theorem 15.3.7, page 295]{BC} to the case of bounded semi-continuous sub/supersolutions under 
three different junction conditions. Even if they are presented for $N = 2$ branches, we present their results in our geometric setting: 
a junction reduced to a single branch $N = 1$ in dimension $d\ge 1$. 
The three cases are the following: (1) $F$ is independent on $p_d$, (2) the Neumann problem and (3) general nonincreasing continuous $p_d\mapsto F(X,p',p_d)$.
 In the third case, the normal derivative is not coupled with the tangential coordinates $(t,x')$ in $F$ (see also the very end of  \cite[Subsection 13.2.2 and condition (GA-G-FLT) p. 247]{BC}).

 As explained above, we improve these results, using the  twin blow-up method introduced in \cite{FIM2}. A close look at the proof reveals that new ideas appear at the beginning of Step 4, when the
 reasoning focuses on the case where the point of maximum is on the boundary of the domain. Compared to the note \cite{FIM2}, we also extend the result by considering uniformly continuous initial data (and not only Lipschitz continuous ones) and working in dimension greater than one.
 
\paragraph{Organization of the paper.}
In Section \ref{s2}, we present two key boundary results  stated for stationary problems in space dimension $d=1$.
We also extend these results to the case of junctions (that will be used in future works).
In Section \ref{s3}, we recall two classical results which are suitable for our purpose. We first construct barriers. We next present some a priori estimates for the sup-convolution of subsolutions to coercive HJ equations.
The proof of the comparison principle in the case of the half space (Theorem \ref{th::b7}) is done in Section \ref{s4}.
Finally in Section \ref{s5}, we show how to adapt our twin blow-up method to the case of a $C^1$ bounded open domain.

\paragraph{Acknowledgements.} The authors thank G. Barles and E. Chasseigne for enlighting discussions during the preparation of this work.
The last author also thanks J. Dolbeault, C. Imbert and T. Leli\`evre for providing him good working conditions.
This research was partially funded by l'Agence Nationale de la Recherche (ANR), project ANR-22-CE40-0010 COSS. 
For the purpose of open access, the authors have applied a CC-BY public copyright licence to any Author Accepted Manuscript (AAM) version arising from this submission.

\section{Boundary lemmas}
\label{s2}

In this section, we work in dimension $d=1$ and set $\Omega:= (0,+\infty)$. We present some fundamental boundary results that will allow us to prove our comparison principle.
At the end of this section, we  also extend them naturally to the case of junctions (that will be useful for future works).

Before to state our result, we need to introduce the following notion of (limiting) semi-differentials.
\begin{defi}[(Limiting) semi-differentials]\label{eq::a1}
Let $A\subset \overline \Omega$ and $x_0\in A$. For $(+/-)$, we define the (first order) super/subdifferential at $x_0$ of a function $u$ on $A$ as
\begin{equation}\label{eq::n3}
D^\pm_A u(x_0)=\left\{p\in \R,\quad \mbox{such that}\quad 0\le \pm\left\{u(x_0)+p\cdot (x-x_0)+o(x-x_0)-u(x)\right\}\quad \mbox{on}\quad A\right\}\
\end{equation}
and the limit (first order) super/subdiffential at the boundary  point $x_0 \in \partial \Omega$ of $u$  as
\begin{equation}\label{eq::n4}
\bar D^\pm_\Omega u(x_0)=\left\{\mbox{$p\in \R$, there exists a sequence $p^k\in D^\pm_\Omega u(x^k)$ with $x^k\in \Omega$ and $(x^k,p^k)\to (x_0,p)$}\right\}.
\end{equation}
\end{defi}

\begin{rem}\label{rem::a2}
Note that if $p\in \bar D^+_\Omega u(x_0)$ with $x_0\in \partial \Omega$, and if $u$ is a subsolution of $H(Du)\le 0$ in $\Omega$, then $H(p)\le 0$.
\end{rem}

We then have the following result.

\begin{lem}[Critical slopes and semi-differentials]\label{lem::n1}
Let $\Omega:=(0,+\infty)$. We consider two functions $u,v: \overline \Omega \to \R\cup \left\{-\infty,+\infty\right\}$ with $u$ upper semi-continuous and $v$ lower semicontinous satisfying $u(0)=0=v(0)$ with $u\le v$ on $\overline \Omega$.
We define the critical slopes by
\begin{equation}\label{eq::n12}
\overline p:=\limsup_{\Omega\ni x \to 0} \frac{u(x)}{x},\quad \underline p:=\liminf_{\Omega\ni x \to 0} \frac{v(x)}{x}.
\end{equation}
Then we have the following (limiting) semi-differential inclusions
\begin{equation}\label{eq::a3}
\R\cap \left[\underline p, \overline p\right]\subset \bar D^+_\Omega u(0) \cap \bar D^-_\Omega v(0)\quad \mbox{if}\quad  \overline p\ge \underline p
\end{equation}
\begin{equation}\label{eq::a4}
\R\cap \left[\overline p,\underline p\right] \subset  D^+_{\overline \Omega} u(0)\cap  D^-_{\overline \Omega} v(0) 
\quad \mbox{if}\quad \overline p \le \underline p \\
\end{equation}
\begin{equation}\label{eq::a5}
\left\{\begin{array}{ll}
\overline p \in \bar D^+_\Omega u(0) &\quad \mbox{if}\quad \overline p \not= -\infty \\
\underline p \in \bar D^-_\Omega u(0) &\quad \mbox{if}\quad \underline p \not= +\infty.
\end{array}\right.
\end{equation}
\end{lem}

\begin{proof}
The proof of this lemma is already contained in \cite{FIM2} but for sake of completeness, we give it here.
We first notice that \eqref{eq::a4} is a straightforward consequence of the  definition of sub and superdifferentials.\newline
In order to prove \eqref{eq::a3}, we first focus on the proof of 
\begin{equation}\label{eq::n13}
\R\cap \left[\underline p, \overline p\right]\subset \bar D^+_\Omega u(0)  \quad \mbox{in case}\quad \overline p > \underline p
\end{equation}
and we will even show the follower better result
\begin{equation}\label{eq::n13bis}
\R\cap \left[\underline q, \overline p\right]\subset \bar D^+_\Omega u(0)  \quad \mbox{in case}\quad \overline p >  \underline q := \liminf_{\Omega\ni x\to 0} \frac{u(x)}{x}.
\end{equation}
Note that  $u\le v$ implies $\underline q\le \underline p$ and so \eqref{eq::n13} is a consequence of \eqref{eq::n13bis}. 
The claim is a variant of (18) in \cite{LS2} and the proof is a variant of  the one done in  Barles, Chasseigne \cite[Lemma 15.3.1]{BC}. We give the details for sake of completeness.
We first assume that  $p\in (\underline q,\overline p)$. This implies that
$$\limsup_{\Omega\ni x\to 0} \frac{u(x)}{x}=\overline p>  p >  \underline q = \liminf_{\Omega\ni x\to 0} \frac{u(x)}{x}$$
and so for any $\varepsilon>0$, there exists $y_\varepsilon \in (0,\varepsilon)$ and $z_\varepsilon \in (0,y_\varepsilon)$ such that
$$\frac{u(z_\varepsilon)}{z_\varepsilon}\  > p >   \frac{u(y_\varepsilon)}{y_\varepsilon}.$$
Hence the function $\zeta(x):=u(x)-px$ satisfies 
$$\zeta(0)=0> \zeta(y_\varepsilon)\quad \mbox{with}\quad M:=\sup_{[0,y_\varepsilon]} \zeta \ge \zeta(z_\varepsilon)>0.$$
Let $x_\varepsilon \in (0,y_\varepsilon)$ be a point of maximum of $\zeta$ in $[0,y_\e]$. We see that the function $x\mapsto px+ M$ is a test function touching $u$ from above at $x_\varepsilon$, which implies that $p\in D^+_\Omega u(x_\varepsilon)$. In the limit $\varepsilon\to 0$, we recover $p\in \bar D^+_\Omega u(0)$  which proves the claim. In the case where $p\in[\underline q, \overline p]$, we get the result by the closedness of $\bar D^+_\Omega u(0)$. This proves \eqref{eq::n13bis}.
A similar inclusion  for $v$ implies  \eqref{eq::a3} in the special case where $\overline p>\underline p$. On the other hand, notice that \eqref{eq::a5} implies \eqref{eq::a3} in the case $\overline p=\underline p$. \medskip

Hence it remains to show \eqref{eq::a5}. We claim that
\begin{equation}\label{eq::n14}
\underline p \in \bar D^-_\Omega v(0) \quad \mbox{if}\quad \underline p \in \R.
\end{equation}
This result is a property of the critical slope for lower semi-continuous functions.
Its proof follows exactly the lines of \cite[Proof of Lemma 2.9]{IM1} (where the  proof does not use any  Hamiltonian).
A similar result holds for $u$ and proves \eqref{eq::a5}.
This ends the proof of the lemma.
\end{proof}

Before to state the fundamental lemma for the comparison principle, we recall the definition of (semi-) coercive functions.
\begin{defi}[Coercive and semi-coercive functions]\label{defi::a7}
Consider a function $G:\R\to \R$. Then $G$ is coercive if $\displaystyle \lim_{|p|\to +\infty} G(p)=+\infty$, and semi-coercive if $\displaystyle \lim_{p\to -\infty} G(p)=+\infty$.
\end{defi}

As a consequence of Lemma \ref{lem::n1}, we have the following result which will be used to prove the comparison principle.
\begin{cor}[Boundary viscosity inequalities]\label{cor::n6}
Let $\Omega$ and $u,v$ be as in statement of Lemma \ref{lem::n1}.
For $\gamma=\alpha,\beta$, consider continuous functions $H_\gamma,F_\gamma: \R\to \R$ with $H_\alpha$  coercive and $F_\alpha$ semi-coercive. Assume that we have the following viscosity inequalities for some $\eta>0$
\begin{equation}\label{eq::n7}
\left\{\begin{array}{rlrl}
H_\alpha(u_x) \le 0 &\quad  \mbox{on} & \quad \Omega &\cap \left\{|u|< +\infty\right\}\\
\min\left\{F_\alpha,H_\alpha\right\}(u_x) \le 0 &\quad  \mbox{on} & \quad \left\{0\right\} &\cap \left\{|u|< +\infty\right\}\\
\\
H_\beta(v_x) \ge  \eta &\quad  \mbox{on} & \quad \Omega &\cap \left\{|v|< +\infty\right\}\\
\max\left\{F_\beta,H_\beta\right\}(v_x) \ge \eta &\quad  \mbox{on} & \quad \left\{0\right\} &\cap \left\{|v|< +\infty\right\}.\\
\end{array}\right.
\end{equation}
For $\underline p,\overline p$ defined in \eqref{eq::n12}, we set $a:=\min \left\{\underline p,\overline p\right\}$ and $b:=\max \left\{\underline p,\overline p\right\}$. Then  $\overline p\in [a,b]\cap \R$ and there exists a real number $p\in [a,b]$ such that
\begin{equation}\label{eq::n8}
\mbox{either}\quad H_\alpha(p)\le 0 < \eta \le  (H_\beta-H_\alpha)(p)\quad 
\mbox{or}\quad \max\left\{F_\alpha,H_\alpha\right\}(p)\le 0 <  \eta \le (F_\beta-F_\alpha)(p).
\end{equation}
\end{cor}
\begin{proof}
The main steps of the proof is given in \cite{FIM2}, but for sake of completeness, we give all the details here.
We begin to explain why $\overline p\in\R$.
Because $H_\alpha$ is coercive and $F_\alpha$ is semi-coercive, we know from \cite[Lemma~3.8]{FIM1} that $u$ is weakly continuous at $x=0$, i.e.
\begin{equation}\label{eq::c4}
0=u(0)=\limsup_{\Omega\ni x\to 0^+} u(x).
\end{equation}
Then  \cite[Proof of Lemma~ 2.10]{IM1} shows additionally that $\overline p > -\infty$. Now we  claim that we also have $\overline p< +\infty$. Indeed,  assume by contradiction that $\overline p=+\infty$. Then, there exists $y_n\to 0$ such that $p_n:=u(y_n)/y_n \to +\infty$. For $b\in \R$, let us define $\phi_b(x):=p_n x+b$ and
$$\overline b =\inf \{b,\; u\le \phi_b \; {\rm in}\; [0,y_n]\}.$$
In particular, there exists $x_n \in [0,y_n]$ such that $\phi_{\overline b}$ touches $u$ from above at $x_n$. If $x_n=0$, then $0=u(0)-\phi_{\overline b}(0)=-\overline b$. In the same way, if $x_n=y_n$, then $u(y_n)= \phi_{\overline b}(y_n)= u(y_n)+\overline b$ and we recover again that $\overline b=0$. This implies that $u(x)\le p_n x$ and so
$$+\infty=\limsup_{x\to 0} \frac {u(x)}x\le p_n<+\infty.$$
We then deduce that $x_n\in (0,y_n)$ and since $u$ is a sub-solution, we get 
$$H_\a(p_n)\le 0$$
which is absurd for $n$ large enough by coercivity of $H_\a$. This implies that $\overline p< +\infty$.
 We conclude that $\overline p\in \R\cap [a,b]$.

We now turn to the proof of \eqref{eq::n8}. If $\underline p \le \overline p$, then \eqref{eq::a3} shows, for all  $p\in \left[\underline p, \overline p\right]\cap \R$, that 
$$H_\alpha(p)\le 0 < \eta\le H_\beta(p)$$ 
which implies in particular the desired conclusion.

 We now assume that  $\underline p > \overline p$. We have  in particular $[a,b]\subset (-\infty,+\infty]$ with $a<b$ and
\begin{equation}\label{eq::n9}
\left\{\begin{array}{rlllll}
H_\alpha(a) \le & 0 &&&\quad \mbox{because}\quad &a\in \R\\
&0 &< \eta &\le H_\beta(b)&\quad \mbox{if}\quad &b\in \R\\
\min\left\{H_\alpha,F_\alpha\right\} \le &0 &< \eta &\le \max\left\{H_\beta,F_\beta\right\} &\quad \mbox{on}\quad   &\left[a,b\right]\cap \R
\end{array}\right.
\end{equation}
where the last line follows from \eqref{eq::a4}, and the first two lines follow from \eqref{eq::a5}.\medskip

We now claim that for all $\varepsilon>0$ small enough, there exists some $p_\varepsilon\in [a,b]\cap \R$ such that we have at $p_\varepsilon$
\begin{equation}\label{eq::01}
\mbox{\bf i)}\quad H_\alpha\le \varepsilon < \eta-\varepsilon \le H_\beta-H_\alpha\quad \mbox{or}\quad \mbox{\bf ii)}\quad \max
\left\{F_\alpha,H_\alpha\right\}\le \varepsilon < \eta \le F_\beta-F_\alpha.
\end{equation}
 By contradiction, we assume that there exists $\varepsilon>0$ (small enough) such that
\begin{equation}\label{eq::n10}
\left\{\begin{array}{ll}
\mbox{\bf i)}&H_\beta-H_\alpha < \eta-\varepsilon\quad \mbox{or}\quad \varepsilon< H_\alpha\\
\mbox{and}&\\
\mbox{\bf ii)}&F_\beta-F_\alpha < \eta \quad \mbox{or}\quad \varepsilon< \max\left\{F_\alpha,H_\alpha\right\}
\end{array}\right|\quad \mbox{for all}\quad p\in [a,b]\cap \R.
\end{equation}
Recall that the coercivity of $H_\alpha$ means $H_\alpha(\pm \infty):=\liminf_{p\to \pm \infty} H_\alpha(p)=+\infty$. We distinguish two cases.\medskip

\noindent {\bf Case 1: $H_\alpha(b)>\varepsilon$}\\
Here $b$ can be finite or equal to $+\infty$. We get 
$$H_\alpha(b)> \varepsilon> 0\ge H_\alpha(a).$$
Therefore by continuity, there exists $p\in (a,b)$ such that $H_\alpha(p)=\varepsilon$.
Hence in the last line of \eqref{eq::n9}, the first inequality implies that $F_\alpha(p)\le 0$.
Because  \eqref{eq::n10} i) and ii) hold true for $p$, we get 
$$H_\beta(p)<\eta\quad \mbox{and}\quad F_\beta(p)<\eta$$
which leads to a contradiction with the second inequality in the last line of \eqref{eq::n9}.\medskip

\noindent {\bf Case 2: $H_\alpha(b)\le \varepsilon$}\\
Then $b\in \R$ and \eqref{eq::n10} i) implies for $p=b$ that $H_\beta(b)< \eta$, which is in contradiction with the second line of \eqref{eq::n9}.\medskip

In all the cases, we get a contradiction, which proves \eqref{eq::01}. 
Since $H_\alpha$ is coercive, we see in both cases i) or ii) of \eqref{eq::01}, that we can always extract a subsequence as $\varepsilon\to 0$ such that 
$p_\varepsilon \to p\in [a,b]\cap \R$. Passing to the limit in \eqref{eq::01}, we get the desired conclusion.
This ends the proof of the corollary.
\end{proof}

Notice that it is very easy to show the following extension of Corollary \ref{cor::n6} to the case of junctions.

\begin{pro}[Junction viscosity inequalities]\label{pro::n6jb}
For $N\ge 1$, let $J^i:=(0,+\infty)$ for $i=1,\dots,N$, and set 
$$J:=\left\{0\right\}\cup \left(\bigcup_{i=1,\dots,N} J^i\right)$$
with the topology of glued branches. For a piecewise $C^1$ function  $u$ on $J$, and $u^i:=u_{|J^i\cup \left\{0\right\}}$, we set
$$u_x(x)=\left\{\begin{array}{ll}
(u^1_x(0),\dots,u^N_x(0))&\quad \mbox{if}\quad x=0\\
u^i_x(x) &\quad \mbox{if}\quad x\in J^i.
\end{array}\right.$$

We consider two sets of functions $u,v: J \to \R\cup \left\{-\infty,+\infty\right\}$ with $u$ upper semi-continuous and $v$ lower semicontinous satisfying
\begin{equation}\label{eq::n2b}
u(0)=0=v(0)\quad \mbox{with}\quad u\le v \quad \mbox{on}\quad J.
\end{equation}
For $i=1,\dots,N$, we define
\begin{equation}\label{eq::n12b}
\overline p^i:=\limsup_{J^i\ni x \to 0} \frac{u(x)}{x},\quad \underline p^i:=\liminf_{J^i\ni x \to 0} \frac{v(x)}{x}.
\end{equation}
We also set  $a^i:=\min \left\{\underline p^i, \overline p^i\right\}$, $b^i:=\max\left\{\underline p^i, \overline p^i\right\}$ and
$$[a,b]\cap \R^N:=\prod_{i=1,\dots,N} [a^i,b^i]\cap \R.$$
For $\gamma=\alpha,\beta$, consider continuous functions $H_\gamma^i: \R\to \R$ and $F_\gamma:\R^N \to \R$ with $H_\alpha^i$ coercive and $F_\a$  semi-coercive. For $p=(p^1,\dots,p^N)\in \R^N$, we set
$$H_{\gamma; \min}(p)=\min_{i=1,\dots,N} H^i_\gamma(p^i),\quad H_{\gamma;\max}(p)=\max_{i=1,\dots,N} H^i_\gamma(p^i).$$
Then assume that we have the following viscosity inequalities for some $\eta>0$
\begin{equation}\label{eq::n7b}
\left\{\begin{array}{rlrll}
H^i_\alpha(u_x) \le 0 &\quad  \mbox{on} & \quad J^i &\cap \left\{|u|< +\infty\right\}& \quad \mbox{for}\quad i=1,\dots,N\\
\min\left\{F_\alpha,H_{\alpha;\min}\right\}(u_x) \le 0 &\quad  \mbox{on} & \quad \left\{0\right\} &\cap \left\{|u|< +\infty\right\}&\\
\\
H_\beta^i(v_x) \ge  \eta &\quad  \mbox{on} & \quad J^i &\cap \left\{|v|< +\infty\right\}& \quad \mbox{for}\quad i=1,\dots,N\\
\max\left\{F_\beta,H_{\beta;\max}\right\}(v_x) \ge \eta &\quad  \mbox{on} & \quad \left\{0\right\} &\cap \left\{|v|< +\infty\right\}.&\\
\end{array}\right.
\end{equation}
Then  there exists $p=(p^1,\dots,p^N)\in [a,b]\cap \R^N\not=\emptyset$ such that
\begin{equation}\label{eq::n8jb}
\left\{\begin{array}{lrcll}
\mbox{either}&\quad H^i_\alpha(p^i)&\le 0 < \eta \le & (H^i_\beta-H^i_\alpha)(p^i) &\quad \mbox{for some $i\in \left\{1,\dots,N\right\}$}\\
\mbox{or}&\quad \max(F_\alpha,H_{\alpha;\max})(p)&\le 0 <  \eta \le &(F_\beta-F_\alpha)(p).&
\end{array}\right.
\end{equation}
\end{pro}

\section{Barriers and regularization}\label{s3}


\begin{lem}[Barriers]\label{lem::60xt}
Let $T>0$ and assume that $H,F$ satisfy \eqref{eq::b4}-\eqref{eq::b5}, and that the initial data $u_0$ is uniformly continuous.
Assume that $u$ (resp. $v$) is an upper semi-continuous subsolution (resp. a  lower semi-continuous supersolution) of \eqref{eq::b1}, \eqref{eq::b3}, satisfying  the a priori bounds \eqref{eq::a10} for some constant $C_T$.

Then there exists a continuous increasing function $f:[0,T] \to [0,+\infty)$ with $f(0)=0$   such that the functions
\[u^\pm(t,x):=u_0(x)\pm f(t)\]
satisfy the following barrier properties:
\begin{itemize}
  \item
if $u\le u_0$ in $\left\{0\right\}\times \overline \Omega$, then $u\le u^+$ in $[0,T)\times \overline \Omega$,
\item
if $v\ge u_0$ in $\left\{0\right\}\times \overline \Omega$, then $v  \ge  u^-$ in $[0,T)\times \overline \Omega$.
\end{itemize}
\end{lem}

\begin{proof}
The idea of the proof is somehow very standard. 
We first extend by continuity the initial data defined on $\R^{d-1}\times [0,+\infty)$ to a function defined on $\R^{d-1}\times \R$, setting
$u_0(x',x):=u_0(x',0)$ for all $x\le 0$.
Hence $u_0$ still satisfies
$$|u_0(x)-u_0(y)|\le \omega_0(|x-y|)\quad \mbox{for all}\quad x,y\in \R^{d}$$
where $\omega_0$ is the modulus of continuity of $(u_0)|_{\R^{d-1}\times [0,+\infty)}$.
We do the proof to compare $u$ and $u^+$, the one to compare $v$ and $u^-$ being similar.\medskip

\noindent {\bf Case 1: $u_0 \in (C^1\cap \mbox{Lip})(\R^{d})$}\\
In this case, there exists some $L>0$ such that  $|D u_0 |_{L^\infty(\R^{d})}\le L$. 
From assumptions i) and ii) of both \eqref{eq::b4} and  \eqref{eq::b5}, we see that 
there exists $\lambda=\lambda(T,L)\ge 0$ minimal such that, for $B_L:=B_L(0)$,
\begin{equation}\label{eq::a11}
-\lambda \le \inf_{[0,T]\times  \overline{\Omega}\times \overline{B}_{2L}} \min\left\{H,F\right\} \le \sup_{[0,T]\times  \overline{\Omega}\times \overline{B}_{2L}} \max\left\{H,F\right\}\le \lambda
\end{equation}
where we have extended the function $F$ as follows: $F(t,x',x_d,p):=F(t,x',0,p)$ for all $x_d\ge 0$.
Setting $f(t):=\lambda t$, we see that $u^+$ is a  supersolution of \eqref{eq::b1}, \eqref{eq::b3}.
Assume now by contradiction that 
$$M:=\sup_{Q_T} (u-u^+)>0\quad \mbox{with}\quad Q_T:=[0,T)\times \overline \Omega.$$
Now  for $\eta,\alpha>0$, let us consider
$$M_{\eta,\alpha}:=\sup_{Q_T} \Phi\quad \mbox{with}\quad \Phi(t,x):=u(t,x)-u^+(t,x)-\frac{\eta}{T-t}-\frac{\alpha}{2} x^2.$$
For  $\eta,\alpha>0$ small enough, we have  $M_{\eta,\alpha} \ge M/2>0$.
Moreover from the bound \eqref{eq::a10} on $u$, we see that  the supremum in $M_{\eta,\alpha}$ is reached for some point $\bar X=(\bar t,\bar x) \in Q_T$. We also have
$$\limsup_{(\eta,\alpha)\to (0,0)} \left\{\frac{\eta}{T-\bar t}+\frac{\alpha}{2} \bar x^2\right\}=0$$
and then we can fix $\eta,\alpha>0$ small enough such that $|\alpha \bar x|\le L$.

Assume that  $\bar t=0$.
Then
$0<M/2\le M_{\eta,\alpha}=\Phi(0,\bar x)\le -\frac{\eta}{T}$
which leads to a contradiction. Hence $\bar t>0$ and 
$\bar X=(\bar t,\bar x)=(\bar t,\bar x',\bar x_d)\in (0,T)\times \R^{d-1}\times [0,+\infty)$. Therefore we have the viscosity inequalities for $p:=D u_0(\bar x)\in \overline{B}_L$ and $\alpha \bar x \in \overline B_L$
$$\left\{\begin{array}{ll}
\ \frac{\eta}{(T-\bar t)^2}+\lambda+ H(\bar X, p+\alpha \bar x) \le 0&\quad \mbox{if}\quad \bar x_d>0\\
\ \frac{\eta}{(T-\bar t)^2}+\lambda+ \max (F,H)(\bar X, p+\alpha \bar x) \le 0&\quad \mbox{if}\quad \bar x_d=0\\
\end{array}\right.$$
which leads to a contradiction from the choice of $\lambda$ in \eqref{eq::a11}. This implies that $M\le 0$ and then $u\le u^+$.\medskip

\noindent {\bf Case 2: $u_0$ is only uniformly continuous}\\
Let $\varphi$ be  a smooth nonnegative function satisfying
$\varphi=0$ on $\R^n\backslash B_1(0)$ and $\int_{\R^n} \varphi(x)dx=1$. For  $\varepsilon>0$, we set the convolution
$u_0^\varepsilon:=\varphi_\varepsilon\star u_0$ with $\varphi_\varepsilon(x)=\frac{1}{\varepsilon^n}\varphi(\frac{x}{\varepsilon})$.
Hence we have
$$|u^\varepsilon_0-u_0|_{L^\infty(\R^n)}\le \omega_0(\varepsilon)$$
and
$D u^\varepsilon_0(x)=\frac{1}{\varepsilon}\int_{\R^n} dy \frac{1}{\varepsilon^n}D \varphi(\frac{y}{\varepsilon}) \left\{u_0(x-y)- u_0(x)\right\}$. Therefore, we get 
$$|D u^\varepsilon_0|_{L^\infty(\R)}\le |D\varphi|_{L^1(\R)}\cdot  \frac{\omega_0(\varepsilon)}{\varepsilon}\le L_\varepsilon:= |D\varphi|_{L^1(\R^n)}\cdot  \sup_{\delta\ge \varepsilon} \frac{\omega_0(\delta)}{\delta}.$$
We  define $\lambda_\varepsilon=\lambda_{T,L_\varepsilon}\ge 0$ minimal such that
$$-\lambda_\varepsilon \le  \inf_{[0,T]\times  \overline{\Omega}\times \overline{B_{2L_\varepsilon}}} \min(H,F) \le \sup_{[0,T]\times  \overline{\Omega}\times \overline{B_{2L_\varepsilon}}} \max(H,F) \le \lambda_\varepsilon,$$
where by construction the map $\varepsilon \mapsto \lambda_\varepsilon$ is nonincreasing, and we set
$f_\varepsilon(t):=\lambda_\varepsilon t$.
Using that $u_0(x)\le u^\varepsilon_0(x)+\omega_0(\varepsilon)$, we can show as in Case 1 that
$u(t,x) \le u^\varepsilon_0(x)+\omega_0(\varepsilon)+\lambda_\varepsilon t\le u_0(x)+2\omega_0(\varepsilon)+\lambda_\varepsilon t$.
If we set
$$f(t):=\inf_{\varepsilon>0} \left\{2\omega_0(\varepsilon)+\lambda_\varepsilon t\right\},$$
where $f$ is a (continuous) concave nondecreasing function satisfying $f(0)=0$, we get
$u\le u^+$. This ends the proof of the lemma.
\end{proof}


We now consider a (classical) regularization of a subsolution $u$ by tangential sup-convolutions. Because we only assume a bound from above  $u\le u_0+C_T$, we have additionally to truncate $u$  from below by some function. We will use the function $\underline u_0:=u_0-C_T$ (which will also be later  in the next section a bound from below for the supersolution $v$), where $u_0$ is the initial data, which is assumed to be Lipschitz continuous, in order to simplify the presentation. Then we have the following result.

\begin{lem}[Tangential regularization  after truncation by Lipschitz initial data]\label{lem::62xt}
Let $T>0$ and assume that $H$ satisfies \eqref{eq::b4} and that the initial data $u_0$ is Lipschitz continuous of Lipschitz constant $L_0$.
Let $u$ be an upper semi-continuous subsolution  of \eqref{eq::b1}-\eqref{eq::b2}, satisfying moreover the a priori bound \eqref{eq::a10} for some constant $C_T$, namely
\begin{equation}\label{eq::a13}
u(t,x)\le u_0(x)+C_T \quad \mbox{for all}\quad (t,x)\in [0,T)\times \overline \Omega.
\end{equation}
We define
$\displaystyle u(T,x):=\limsup_{(s,y)\to (T,x),\ s<T} u(s,y)$ for all $x\in \overline \Omega=\R^{d-1}\times [0,+\infty)$, 
and extend $u$ to $\R\times \overline \Omega$, setting
$u(t,x):=u(T,x)$ if $t\ge T$, and 
$u(t,x):=u(0,x)$ if $t\le 0$.
We set
\[\tilde u:=\max(u, \underline u_0)\quad \mbox{with}\quad \underline  u_0(t,x):=u_0(x)-C_T.\]
We denote the tangential variable by $\xi=(\xi^0,\xi')=(s,x')\in\R^d$ and the normal variable by $x_d\in [0,+\infty)$ and we define for $\nu>0$ the tangential sup-convolution
\[\tilde u^\nu(\xi,x_d):=\sup_{\zeta \in \R^d} \left\{\tilde u(\zeta,x_d)-\frac{|\xi-\zeta|^2}{2\nu}\right\}=\tilde u(\bar \zeta,x_d)-\frac{|\xi-\bar \zeta|^2}{2\nu}\]
where each $\bar \zeta$ depends on $(\xi,x_d)\in \R^d\times [0,+\infty)$ with  $|\overline \zeta -\xi|\le \theta^\nu:=\sqrt{5(4\nu C_T+ \nu^2L_0^2)}< T/2$, for $\nu$ small enough.

Then the function $\tilde u^\nu$ is Lipschitz continuous in $\R\times \overline \Omega$ with respect to the variable $\xi$. Moreover it is Lipschitz continuous  in $I^\nu\times \Omega$ with respect to the variable $x_d$,  with $I^\nu:=(\theta^\nu,T-\theta^\nu)$,
$$|D_\xi \tilde u^\nu|_{L^\infty(\R\times \overline \Omega)} \le \frac{\theta^\nu}{\nu}\quad \mbox{and}\quad  |\partial_{x_d} \tilde u^\nu|_{L^\infty(I^\nu\times \Omega)} \le \max\left\{L^\nu,L_0\right\}$$
where 
$L^\nu:=\sup \left\{p_d\in \R,\quad \inf_{(X,p')\in ([0,T]\times \overline \Omega)\times \overline B_{\frac{\theta^\nu}{\nu}}} H(X,p',p_d) \le \frac{\theta^\nu}{\nu}\right\}$.

Assume furthermore that $u$ is a subsolution at the boundary $(0,T)\times \partial \Omega$, i.e. satisfies the first line of \eqref{eq::b3} for some $F$ satisfying \eqref{eq::b5}. Then $\tilde u^\nu$ is Lipschitz continuous in space and time on $\overline {I^\nu\times  \Omega}$ of Lipschitz constant $L_\nu:=\max\left\{\frac{\theta^\nu}{\nu},L^\nu,L_0\right\}$.

\end{lem}

\begin{proof}
The proof is splited into three steps.

\noindent {\bf Step 1: first bounds using the $2$-sided bound}\\
We begin to show that
\begin{equation}\label{eq:n1}
|\bar  \zeta -\xi|\le \theta^\nu.
\end{equation}
From the $1$-sided bound  \eqref{eq::a13} and the definition of $\underline  u_0$ and $\tilde u$, 
we get the $2$-sided bound
\begin{equation}\label{eq::a14}
|\tilde u(t,x)-u_0(x)|\le C_T \quad \mbox{for all}\quad (t,x)\in \R\times \overline \Omega.
\end{equation}
For $\xi=(t,\xi'):=(t,x')$ and $z:=x_d$ (to simplify the notation), we have
$$\tilde u(\xi,z)\le \tilde u^\nu(\xi,z) = \tilde u(\bar \zeta,z)-\frac{|\xi-\bar \zeta|^2}{2\nu}$$
with $\bar \zeta=(\bar t,\bar \zeta')$. Using \eqref{eq::a14}, we then get
$$\frac{|\xi-\bar \zeta|^2}{2\nu} -2C_T\le u_0(\bar \zeta',z)-u_0(\xi',z)\le L_0|\bar \zeta'-\xi'|.$$
This implies  
$$\displaystyle \left\{|\xi'-\bar \zeta '|-\nu L_0\right\}^2+|t-\bar t|^2 \le 4\nu C_T+ \nu^2L_0^2=\frac{(\theta^\nu)^2}{5}.$$
We then deduce that $|t-\bar t|\le \frac{\theta^\nu}{\sqrt 5}$, $|\xi'-\bar \zeta '|\le 2 \frac{\theta^\nu}{\sqrt 5}$ which implies \eqref{eq:n1}.\medskip

We now prove that 
\begin{equation}\label{eq:n2}
|D_\xi \tilde u^\nu|_{L^\infty(\R\times \overline \Omega)} \le \frac{\theta^\nu}{\nu}.
\end{equation}
For $\xi^a\in \R^d$, we set
$$\tilde u^\nu(\xi^a,z):=\sup_{\zeta\in \R^d} \left\{\tilde u(\zeta,z)-\frac{|\xi^a-\zeta|^2}{2\nu}\right\}=\tilde u(\bar \zeta^a,z)-\frac{|\xi^a-\bar \zeta^a|^2}{2\nu}.$$
Hence, by definition, we have
$$\tilde u^\nu(\xi^a,z) \ge \tilde u(\bar \zeta,z)-\frac{|(\xi^a-\xi)+ \xi-\bar \zeta|^2}{2\nu}=\tilde u^\nu(\xi,z)-(\xi^a-\xi) \cdot \frac{(\xi-\bar \zeta)}{\nu}-\frac{|\xi^a-\xi|^2}{2\nu}$$
and also by symmetry
$$\tilde u^\nu(\xi,z) \ge \tilde u^\nu(\xi^a,z)-(\xi-\xi^a) \cdot \frac{(\xi^a-\bar \zeta^a)}{\nu}-\frac{|\xi-\xi^a|^2}{2\nu}$$
i.e.
$$\frac{|\tilde u^\nu(\xi^a,x)-\tilde u^\nu(\xi,x)|}{|\xi^a-\xi|} \le \max\left\{\frac{|\xi-\bar \zeta|}{\nu},\frac{|\xi^a-\bar \zeta^a|}{\nu}\right\} + \frac{|\xi^a-\xi|}{2\nu}\le \frac {\theta^\nu}\nu+ \frac{|\xi^a-\xi|}{2\nu}. $$
This implies that $\tilde u^\nu$ is Lipschitz continuous in the tangential coordinates and \eqref{eq:n2}.\medskip

\noindent {\bf Step 2: bounds on the normal gradient}\\
It is easy to check that $\tilde u^\nu$ is upper semi-continuous (because this is the case for $u$ itself and the supremum in $\xi$ is locally taken in a compact set).
Let $\varphi$ be a test function touching $\tilde u^\nu$ from above at $X_0:=(t_0,x_0)\in (I^\nu\times  \Omega) \cap \left\{\tilde u^\nu> \underline  u_0^\nu\right\}$ and set $\xi_0:=(t_0,x_0')$ and $z_0:=(x_0)_d$.
We have
$$\underline  u_0^\nu(X_0)<\tilde u^\nu(X_0):=\sup_{h\in \R^d} \left\{\tilde u(\xi_0+h,z_0)-\frac{|h|^2}{2\nu}\right\}=\tilde u(\xi_0+\bar h_0,z_0)-\frac{|\bar h_0|^2}{2\nu}\quad \mbox{with}\quad |\bar h_0|\le \theta^\nu$$
and
$$\tilde u(\xi+\bar h_0,z)-\frac{|\bar h_0|^2}{2\nu} \le \tilde u^\nu(\xi,z) \le \varphi(\xi,z)\quad \mbox{with equality at $(\xi,z)=(\xi_0,z_0)$}.$$
Setting
$$\bar \varphi(\xi,z):=\varphi(\xi-\bar h_0,z)+\frac{|\bar h_0|^2}{2\nu},\quad \bar X_0:=(\xi_0+\bar h_0,z_0),$$
we then get
$$\left\{\begin{array}{l}
\tilde u \le \bar \varphi \quad \mbox{with equality at $\bar X_0$}\\
\\
 \tilde u(\bar {X_0})= \tilde u^\nu({X_0})+\frac{|\bar h_0|^2}{2\nu} >  \underline u_0^\nu({X_0})+\frac{|\bar h_0|^2}{2\nu}= \sup_{h\in \R^d}\left\{\underline u_0(\xi_0+h,z_0)-\frac{|h|^2}{2\nu} \right\}+\frac{|\bar h_0|^2}{2\nu}\ge \underline u_0(\bar {X_0}).$$
\end{array}\right.$$
Hence
$\tilde u=u$ at $\bar {X_0} \in \left([-|\bar h_0|,|\bar h_0|]+I^\nu\right)\times {\Omega}\subset (0,T)\times {\Omega}$.
Because $u\le \tilde u$, we have that $\bar \varphi$ touches $u$ from above at $\bar {X_0}$ and since $u$ satisfies the viscosity inequalities on $(0,T)\times {\Omega}$, we get
$$
\bar \varphi_t(\bar {X_0})+H(\bar X_0,D\bar \varphi(\bar {X_0}))\le 0$$
i.e.
$$
\varphi_t({X_0})+H(\bar X_0,D\varphi({X_0}))\le 0.$$
Setting for $\theta>0$ and $X=(t,x',x_d)$
$$H_\theta(X,p):=\min_{h\in \overline B_{\theta}} H((t,x')+h,x_d,p),$$
we see that we have
$$
\varphi_t({X_0})+H_{\theta^\nu}(X_0,D\varphi({X_0}))\le 0$$
which shows that $\tilde u^\nu$ satisfies this viscosity inequality on $(I^\nu\times  \Omega)\cap \left\{\tilde u^\nu> \underline  u_0^\nu\right\}$.

Recall that we have $\tilde u\ge \underline u_0$ and so $\tilde u^\nu\ge \underline u_0^\nu$ with $|D\underline u_0^\nu|\le |D\underline u_0|\le L_0$.
Hence we get in the viscosity sense
$$\min\left\{-\frac{\theta^\nu}{\nu}+H_{\theta^\nu}(X,D \tilde u^\nu),\ |D \tilde u^\nu| - L_0\right\} \le 0 \quad \mbox{on}\quad I^\nu \times \Omega$$
which implies that
$$|\partial_{x_d} \tilde u^\nu| \le \max\left\{L^\nu,L_0\right\} \quad \mbox{on}\quad I^\nu \times \Omega.$$
\medskip

\noindent {\bf Step 4: Lispchitz bounds on $[\theta^\nu,T-\theta^\nu]\times \overline \Omega$}\\
Because $H$ is coercive and $F$ is semi-coercive, we can use \cite[ Lemma 3.8]{FIM1} to get  that $u$ is weakly continuous for all $(t,x)\in (0,T)\times \partial \Omega$, i.e.
$$u^*(t,x)=\limsup_{(s,y) \to (t,x),\quad y\in \Omega} u(s,y)$$
which is again the case for $\tilde u=\max\left\{u,\underline u_0\right\}$. By sup-convolution, it is then easy to check that this is also true for
$$\tilde u^\nu(\xi,x_d)=\sup_{|\zeta-\xi| \le \theta^\nu} \left\{\tilde u(\zeta,x_d)-\frac{|\xi-\zeta|^2}{2\nu}\right\}$$
at least for all $t\in I^\nu=(\theta^\nu,T-\theta^\nu)$. Because $\tilde u^\nu$ is uniformly Lipschitz continuous on $I^\nu\times \Omega$, we deduce that $\tilde u^\nu$ is also Lipschitz continuous on $I^\nu\times \overline \Omega$. Finally, because the bound on $\partial_t \tilde u^\nu$ is uniform in space and time, we deduce that $\tilde u^\nu$ is Lipschitz continuous on $[\theta^\nu,T-\theta^\nu]\times \overline \Omega$ with the same Lipschitz constants. This ends the proof of the lemma.
\end{proof}

\section{The comparison principle on a half space}
\label{s:comp-gen}\label{s4}

This Section is devoted to the proof of the comparison principle Theorem~\ref{th::b7}.
\begin{proof}[Proof of Theorem~\ref{th::b7}]
The strategy of the proof is similar to the one of the note \cite{FIM2} but need technical adaptations. We first follow the proof of the comparison principle in \cite{IM1}, but then modify the proof on the boundary, introducing the twin blow-ups method.
Let $\eta,\theta>0$ and consider
\begin{equation}\label{eq::c122xt}
M(\theta):=\sup \left\{\Psi(\xi,\zeta,x_d),\quad x_d\in [0,+\infty),\quad \xi,\zeta \in [0,T)\times \R^{n-1},\quad |\xi-\zeta|\le \theta\right\}
\end{equation}
with $\zeta=(s,x')$, $x=(x',x_d)$ and
$$\Psi(\xi,\zeta,x_d):=\tilde u(\xi,x_d)-v(\zeta,x_d)-\frac{\eta}{T-s},\quad \tilde u=\max\left\{u,\underline u_0\right\},\quad \underline u_0(t,x):=u_0(X)-C_T=:\underline u_0(x)$$
where we choose carefully $\ \frac{\eta}{T-s}$ instead of $\ \frac{\eta}{T-t}$, because we want to do later a doubling of variables which looks like a sup-convolution (in particular in time) to the function $\tilde u$.\\
We want to prove that
$M:=\lim_{\theta \to 0} M(\theta)\le 0$.
Assume by contradiction that
\begin{equation}\label{eq::c152}
M>0.
\end{equation}
\noindent {\bf Step 0. Reduction to $u_0$ Lipschitz continuous}\\
By assumption, the initial data $u_0$ is uniformly continuous.
We follow the line of Case 2 of the proof of Lemma \ref{lem::60xt}.
We first extend $u_0$ by the value  $u_0(x',0)$ for $z=x_d\le 0$ and $x'\in \R^{n-1}$.
For the ball  $B_1=B_1(0)$, we then consider a smooth nonnegative function $\varphi$ satisfying 
$\varphi=0$ on $\R^n\backslash B_1$ with $\int_{\R^n} \varphi(x)dx=1$,
and for  $\beta>0$, we set the convolution
$u_{0,\beta}:=\varphi_\beta\star u_0$ with $\varphi_\beta(x)=\frac{1}{\beta^n}\varphi(\frac{x}{\beta})$.
Then we can insure that
$$|u_{0,\beta}-u_0|\le \omega_0(\beta),\quad |D u_{0,\beta}| \le L_\beta$$
where $\omega_0$ is the modulus of continuity of $u_0$ and $L_\beta$ is some constant.\\
We have in particular 
$$u-\omega_0(\beta) \le u_{0,\beta} \le v+ \omega_0(\beta)\quad \mbox{on}\quad \left\{0\right\}_t\times \overline \Omega.$$
Hence the problem is reduced to  replace  the quantities $(u,u_0,v)$  by $(u-\omega_0(\beta),u_{0,\beta},v+\omega_0(\beta)$ and $M$  by $M-2\omega_0(\beta)$, and keep $C_T$ unchanged.
Therefore fixing some  $\beta:=\beta_1>0$  small enough such that 
$2\omega_0(\beta)< M$, 
we see that we can redefine $u,u_0,v$ and assume without loss of generality that $u_0$ is Lipschitz continuous, say  (forgetting now $\beta$) for some Lipschitz constant $L_0$, with $M>0$.
\medskip

\noindent {\bf Step 1. Doubling naively the space variables}\\
We first consider a space penalization and standard doubling of variables in space and time, but we distinguish the tangential variables from the normal variable. To this end, we introduce parameters  $\alpha,\nu,\delta>0$ and set
$$M_{\nu,\alpha,\delta}:=\sup_{(t,x),(s,y)\in [0,T)\times \overline \Omega} \Psi_{\nu,\alpha,\delta}(t,x,s,y)$$
with for $\xi=(t,x')$, $\zeta=(s,y')$, $x=(x',x_d)$, $y=(y',y_d)$
\begin{equation}\label{eq::a24}
\Psi_{\nu,\alpha,\delta}(t,x,s,y):=\tilde u(t,x)-v(s,y)-\alpha g(y)-\frac{\eta}{T-s}-\frac{|(t,x')-(s,y')|^2}{2\nu}-\frac{|x_d-y_d|^2}{2\delta},\quad g(y):=\frac{y^2}{2}
\end{equation}
which satisfies
$\displaystyle \liminf_{\alpha\to 0} \left\{\liminf_{\delta\to 0} M_{\nu,\alpha,\delta}\right\} \ge M>0$.
Hence we see that (independently on $\nu>0$) for $\alpha>0$ small enough, and  for  $\delta>0$  small enough (say $\delta\in (0,\delta_\alpha]$), we get
\begin{equation}\label{eq::c120xt}
M_{\nu,\alpha,\delta} \ge M/2>0.
\end{equation}
In particular, the maximum is reached at some point that $(\bar X_\d,\bar Y_\d)=(\bar t_\delta, \bar x_\delta, \bar s_\d, \bar y_\d)$ and we claim that we have the following estimate. 
\begin{lem}[Bounds on any optimizing sequence]\label{lem::a15}
Given $T,C_T>0$, there exists $\eta>0$ small enough such that the following holds true.
Let $X,Y\in [0,T)\times \overline \Omega$ with $X=(t,x)$, $Y=(s,y)$ be such that 
$$0<\Psi_{\nu,\alpha,\delta}(X,Y).$$
Then for $\nu,\delta>0$ small enough (depending on $\eta>0$), we have
$$\left\{\begin{array}{l}
\displaystyle t,s \in [\tau_\eta,T-\tau_\eta],\quad \tau_\eta:=\frac{\eta}{4 C_T}\\
\displaystyle \alpha g(y)+\frac{\eta}{T-s} +\frac{|t-s|^2}{2\nu} \le 3 C_T\\
\displaystyle |x'-y'|\le \sqrt{2\nu\left\{C_T+\frac{\delta L_0^2}{4}\right\}}+2\nu L_0\\
\displaystyle |x_d-y_d|\le \sqrt{2\delta\left\{C_T+\frac{\nu L_0^2}{4}\right\}}+2\delta L_0.\\
\end{array}\right.$$
\end{lem}
This result is standard but since we need precise constants in the estimation, we postponed the proof.\medskip

\noindent {\bf Step 2. When the doubled normal variable converges to a single variable}\\
Step 1 shows that up to extract a subsequence, we have for $B_{\rho_\alpha}=B_{\rho_\alpha}(0)$ and for $\nu$ small enough,
\begin{equation}\label{eq::d7}
(\bar X_\delta,\bar Y_\delta) \to (\bar X,\bar Y)\in \left([\tau_\eta,T-\tau_\eta]\times \overline{B}_{\rho_\alpha}\right)^2\quad \mbox{as}\quad \delta\to 0\quad \mbox{where}\quad \rho_\alpha:=2\sqrt{\frac{6C_T}{\alpha}},
\end{equation}
with $\bar X=(\bar t,\bar x',\bar x_d)=(\bar \xi, \bar x_d)$, $\bar Y=(\bar s,\bar y',\bar y_d)=(\bar \zeta, \bar y_d)$,
$\bar y_d=\bar x_d$, and
$$\alpha g(\bar y)+\frac{\eta}{T-\bar s}+\frac{|\bar t-\bar s|^2}{2\nu}\le 3C_T\quad \mbox{and}\quad |\bar x'-\bar y'| \le \sqrt{2\nu C_T}+2\nu L_0=o_{\nu}(1)\to 0\quad \mbox{as}\quad \nu \to 0$$
where the last bound follows from estimate of Lemma \ref{lem::a15}. Moreover, we have
$$0<M/2\le M_{\nu,\alpha,\delta}\to M_{\nu,\alpha}:=\sup_{X,Y\in [0,T)\times \overline \Omega} \Psi_{\nu,\alpha}(X,Y)=\Psi_{\nu,\alpha}(\bar X,\bar Y)\quad \mbox{as}\quad \delta\to 0$$
with
$$\Psi_{\nu,\alpha}(t,x,s,y):=\left\{\begin{array}{ll}
\displaystyle \tilde u(t,x',x_d)-v(s,y',x_d)-\alpha g(y)-\frac{\eta}{T-s} -\frac{|(t,x')-(s,y')|^2}{2\nu}&\quad \mbox{if}\quad x_d=y_d\\
-\infty &\quad \mbox{if}\quad x_d\not=y_d.\\
\end{array}\right.$$
From the fact that $M_{\nu,\alpha}\to M_{\nu,0}$ as $\alpha\to 0$ (with obvious definitions), we deduce that all maximizer in the definition of $M_{\nu,\alpha}$ satisfies
\begin{equation}\label{eq::d11}
\lim_{\alpha\to 0} \alpha g(\bar y) =0.
\end{equation}
Moreover we have
$$\lim_{\nu\to 0} \left(\lim_{\alpha\to 0} M_{\nu,\alpha}\right) =M=\lim_{\theta \to 0} M(\theta)$$
where $M(\theta)$ is defined in \eqref{eq::c122xt}.
This also implies that
\begin{equation}\label{eq::c123xt}
\lim_{\nu \to 0} \left(\lim_{\alpha\to 0} \frac{|(\bar t,\bar x')-(\bar s,\bar y')|^2}{\nu}\right)=0.
\end{equation}
\medskip

We now prove that $\bar X \in \left\{\tilde u>\underline u_0\right\}$. 
Assume by contradiction that
$\tilde u(\bar X)=\underline u_0(\bar X).$ Then
$$\Psi_{\nu,\alpha}(\bar X,\bar Y)\le \underline u_0(\bar X)-v(\bar Y)-\frac{|\bar x'-\bar y'|^2}{2\nu}$$
which implies (using the a priori  bound $v\ge \underline u_0$)
\begin{equation}\label{eq::a18}
\begin{array}{ll}
0<M/4 &\le M_{\nu,\alpha}=\Psi_{\nu,\alpha}(\bar X,\bar Y)\\
& \displaystyle \le \underline u_0(\bar x)-\underline u_0(\bar y)-\frac{|\bar x'-\bar y'|^2}{2\nu}\\
&\displaystyle \le L_0|\bar x'-\bar y'| -\frac{|\bar x'-\bar y'|^2}{2\nu}.
\end{array}
\end{equation}
This leads to a contradiction as $\nu\to 0$.
Hence
\begin{equation}\label{eq::a19}
\bar X \in \left\{\tilde u>\underline u_0\right\}.
\end{equation}
This implies also that for $\delta$ small enough $\bar X_\d \in \left\{\tilde u>\underline u_0\right\}.$

\medskip

\noindent {\bf Step 3: proof that $\bar x_d=0$}\\
By contradiction, we assume that we are in the standard case $\bar x_d>0$. 
Then we also have $(\bar x_\delta)_d,(\bar y_\delta)_d>0$ and the viscosity inequalities with $\ \bar p_\delta:=\left(\frac{\bar x'_\delta-\bar y'_\delta}{\nu},\frac{(\bar x_\delta)_d-(\bar y_\delta)_d}{\delta}\right)$
\begin{equation}\label{eq::a17}
\left\{\begin{array}{rll}
\frac{\bar t_\delta-\bar s_\delta}{\nu}+ H(\bar X_\delta, \bar p_\delta) &\le 0&\quad \mbox{because}\quad \bar X_\delta\in \left\{\tilde u>\underline  u_0\right\}\\
\\
-\frac{\eta}{(T-\bar s_\delta)^2}+  \frac{\bar t_\delta-\bar s_\delta}{\nu}+ H(\bar Y_\delta, -\alpha Dg(\bar y_\delta)+\bar p_\delta) &\ge 0.&\\
\end{array}\right.
\end{equation}
We know from  Lemma \ref{lem::a15}, that
$$\left|\frac{\bar x'_\delta-\bar y'_\delta}{\nu}\right|\le \nu^{-1}\left\{\sqrt{2\nu\left\{C_T+\frac{\delta L_0^2}{4}\right\}}+2\nu L_0\right\}$$
and
$$H(\bar X_\delta, \bar p_\delta)  \le -\left\{\frac{\bar t_\delta-\bar s_\delta}{\nu}\right\}\le \sqrt{\frac{6 C_T}{\nu}}.$$
Moreover, the uniform coercivity of $H$ (see \eqref{eq::b4} iv)) implies the existence of some $\tilde L_\nu>0$ (independent on $\delta$, for $\delta>0$ small enough, and independent on $\alpha$) such that
$$\left|\left(\frac{\bar t_\delta-\bar s_\delta}{\nu},\bar p_\delta \right)\right| \le \tilde L_\nu.$$
We can then subtract the two viscosity inequalities in \eqref{eq::a17}, and get
$$\frac{\eta}{(T-\bar s_\delta)^2} \le H(\bar Y_\delta, -\alpha Dg(\bar y_\delta)+\bar p_\delta) -H(\bar X_\delta, \bar p_\delta).$$
Passing to the limit $\delta\to 0$, we get (up to extraction of a subsequence) that $\left(\frac{\bar t_\delta-\bar s_\delta}{\nu},\bar p_\delta\right)\to \left(\frac{\bar t-\bar s}{\nu},\bar p\right)$ with 
$$\bar p=\left(\frac{\bar x'-\bar y'}{\nu},\bar p_d\right)\in \R^{d-1}\times \R\quad \mbox{and}\quad \left(\frac{\bar t-\bar s}{\nu},\bar p\right)\in \bar D^{1,+}_{t,x}u(\bar X) \quad \mbox{with}\quad \left|\left(\frac{\bar t-\bar s}{\nu},\bar p \right)\right| \le \tilde L_\nu$$
and
\begin{equation}\label{eq::a20}
\frac{\eta}{(T-\bar s)^2} \le H(\bar Y, -\alpha Dg(\bar y)+\bar p) -H(\bar X, \bar p)\quad \mbox{with}\quad \bar y_d=\bar x_d.
\end{equation}
Using assumptions \eqref{eq::b4} ii) and iii), this implies that (say with $L:= 2 \tilde L_\nu$ and $\alpha$ small enough)
$$\frac \eta {T^2}\le \omega_L(\a Dg(\bar y))+\omega(|\bar X-\bar Y|(1+|\bar p'|+\max(0,H(\bar X,\bar p))\le  \omega_L(\a Dg(\bar y))+\omega(|\bar X-\bar Y|(1+\frac{|\bar x'-\bar y'|}{\nu}+\frac{|\bar t-\bar s|}{\nu})).$$
Using the fact that $\alpha Dg(\bar y) \to 0$ as $\alpha\to 0$ and estimate \eqref{eq::c123xt}, we get a contradiction for $\a$ and $\nu$ small enough.\medskip

\noindent {\bf Step 4: The key one-sided Lipschitz estimate}\\
In the remaining of the proof we then have $\bar x_d=0$.
For $\xi=(t,x')$, $\zeta=(s,y')$, and $x_d\in [0,+\infty)$, we set
$$\Psi^\alpha_\nu(\xi,\zeta,x_d):=\Psi_{\nu,\alpha}(\xi,x_d,\zeta,x_d)=\tilde u(\xi,x_d)-v(\zeta,x_d)-\frac{\eta}{T-s} -\alpha g(y',x_d)-\frac{|\xi-\zeta|^2}{2\nu}$$
and consider
\begin{equation}\label{eq::d6}
0<M/2\le M_{\nu,\alpha}=\sup_{\xi,\zeta \in ([0,T)\times \R^{d-1})^2,\ x_d\in  [0,+\infty)} \Psi^\alpha_\nu(\xi,\zeta,x_d)=\Psi^\alpha_\nu(\bar \xi,\bar \zeta,\bar x_d).
\end{equation}

We define
$$V(s,y):=v(s,y)+\frac{\eta}{T-s}+\alpha g(y)$$
so that we have
$$\left\{\begin{array}{rll}
\partial_t \tilde u+H(X,D \tilde u)\le 0 &\quad \mbox{in}\quad ((0,T)\times \Omega) &\cap \left\{\tilde u>\underline u_0\right\}\\
\\
\partial_t \tilde u+\min\left\{F,H\right\}(X,D \tilde u)\le 0 &\quad \mbox{on}\quad ((0,T)\times \partial \Omega) &\cap \left\{\tilde u>\underline u_0\right\}\\
\\
-\frac{\eta}{(T-s)^2}+\partial_s V+H(Y,D V-\alpha Dg)\ge 0 &\quad \mbox{in}\quad (0,T)\times \Omega&\\
\\
-\frac{\eta}{(T-s)^2}+\partial_s V+\max\left\{F,H\right\}(Y,D V-\alpha Dg)\ge 0 &\quad \mbox{on}\quad (0,T)\times \partial \Omega.&\\
\end{array}\right.$$
We now claim the following one-sided "Lipschitz" estimate
\begin{equation}\label{eq::a23}
\tilde u(\xi,x_d)-V(\zeta,y_d)\le \tilde u(\bar \xi,\bar x_d)-V(\bar \zeta,\bar x_d)+\frac{|\xi-\zeta|^2}{2\nu}-\frac{|\bar \xi-\bar \zeta|^2}{2\nu}+ L_\nu|x_d-y_d|
\end{equation}
where $L_\nu$ is given in Lemma \ref{lem::62xt}, and 
where equality holds for $t=\bar t$, $s=\bar s$, $x'=\bar x'$, $y'=\bar y'$ and $x_d=y_d = \bar x_d$, with $\bar X=(\bar \xi,\bar x_d)$, $\bar Y=(\bar \zeta,\bar x_d)$. For clarity, the proof of \eqref{eq::a23} is postponed at the end of the proof of the theorem.\medskip

\noindent {\bf Step 5: the  twin blow-ups.}\\
We then consider the following  twin blow-ups with small parameter $\varepsilon>0$: one blow-up for  $\tilde u$ at the point $\bar X=(\bar \xi,\bar x_d)$ and one blow-up for  $V$ at the point $\bar Y=(\bar \zeta,\bar x_d)$,
\begin{equation}\label{eq::n27}
\begin{cases}
U^{\varepsilon}(\hat X):=\varepsilon^{-1} \left\{  { \tilde u(\bar X+\varepsilon \hat X)- \tilde u (\bar X)}\right\},&\quad U^{\varepsilon}(0)=0,\\
V^{\varepsilon}(\hat Y):=\varepsilon^{-1} \left\{V(\bar Y+\varepsilon \hat Y)-V(\bar Y)\right\},&\quad V^{\varepsilon}(0)=0.
\end{cases}
\end{equation}
Before passing to the limit $\varepsilon \to 0,$ they satisfy for $\hat X=(\hat t,\hat x)$ and $\hat Y=(\hat s,\hat y)$
\begin{equation}\label{lequation}
\left\{\begin{array}{rll}
\partial_{\hat t} U^\varepsilon + H(\bar X+\varepsilon \hat X,D_{\hat x} U^\varepsilon) \le 0 &\quad \text{in}&\quad I^\varepsilon_{\bar t}\times \Omega\\
\partial_{\hat t} U^\varepsilon + \min(F,H)(\bar X+\varepsilon \hat X,D_{\hat x} U^\varepsilon) \le 0 &\quad \text{on}&\quad I^\varepsilon_{\bar t}\times \partial \Omega\\
\\
-\bar \eta^\varepsilon+\partial_{\hat s} V^\varepsilon + H(\bar Y+ \varepsilon \hat Y, D_{\hat y} V^\varepsilon-\alpha D_y g(\bar y+\varepsilon \hat y)) \ge 0 &\quad \text{in}&\quad I_{\bar s}^\varepsilon\times \Omega\\
-\bar \eta^\varepsilon+\partial_{\hat s} V^\varepsilon + \max(F,H)(\bar Y+ \varepsilon \hat Y,D_{\hat y} V^\varepsilon -\alpha D_{y} g(\bar y+\varepsilon \hat y)) \ge 0 &\quad \text{on}&\quad I_{\bar s}^\varepsilon\times \partial\Omega\\
\end{array}\right.
\end{equation}
with
\[
  \bar \eta^\varepsilon (\hat s):=\frac{\eta}{(T-(\bar s+ \varepsilon \hat s))^2}\quad \text{and}\quad {I_{\bar r}^\varepsilon:= \left(-\frac{\bar r}{\varepsilon},\frac{T-\bar r}{\varepsilon}\right)\quad \mbox{for}\quad \bar r=\bar t,\bar s.}
\]
{From \eqref{eq::a23}, they also satisfy
\begin{equation}\label{eq::Keyb}
U^\varepsilon(\hat \xi,\hat x_d)-V^\varepsilon(\hat \zeta,\hat y_d) \le L_\nu|\hat x_d-\hat y_d|+\bar b \cdot (\hat \xi-\hat \zeta)+\varepsilon\frac{|\hat \xi -\hat \zeta|^2}{2\nu}\quad \mbox{with}\quad \bar b:=\frac{\bar \xi-\bar \zeta}{\nu}.
\end{equation}

We then define  the following half-relaxed limits
$$\left\{\begin{array}{ll}
\displaystyle U^0:=\limsup_{\varepsilon\to 0}{}^* U^\varepsilon,&\quad U^0(0)\ge 0,\\
\displaystyle V^0:=\liminf_{\varepsilon\to 0}{}_* V^\varepsilon,&\quad V^0(0)\le 0.
\end{array}\right.$$
Passing to the limit in \eqref{eq::Keyb}, we get
\begin{equation}\label{eq::n31}
{U^0(\hat \xi,\hat x_d)-V^0(\hat \zeta,\hat y_d)\le L_\nu|\hat x_d-\hat y_d|+\bar b\cdot (\hat \xi-\hat \zeta)},\end{equation}
which implies in particular that $U^0(0)=0=V^0(0).$
Passing to the limit in \eqref{lequation} and using the discontinuous stability of viscosity solutions, we also get}
\begin{equation}\label{eq::n32}
\left\{\begin{array}{rll}
\partial_{\hat t}U^0 + H(\bar X,D_{\hat x}U^0) \le 0 &\quad \text{in}&\quad (\R\times \Omega)\cap \left\{|U^0|< +\infty\right\}\\
\partial_{\hat t} U^0 + \min(F,H)(\bar X,D_{\hat x} U^0) \le 0 &\quad \text{on}&\quad (\R\times \partial \Omega)\cap \left\{|U^0|< +\infty\right\}\\
\\
-\bar \eta+\partial_{\hat s} V^0 + H(\bar Y,D_{\hat y}V^0-\alpha D_y g(\bar y)) \ge 0 &\quad \text{in}&\quad (\R\times \Omega)\cap \left\{|V^0|< +\infty\right\}\\
-\bar \eta+\partial_{\hat s} V^0 + \max(F,H)(\bar Y,D_{\hat y} V^0-\alpha D_y g(\bar y)) \ge 0 &\quad \text{on}&\quad  (\R\times \partial \Omega)\cap \left\{|V^0|< +\infty\right\}
\end{array}\right.
\end{equation}
with $\bar \eta:=\frac{\eta}{(T-\bar s)^2}$.

\medskip

\noindent \textbf{Step 6: the 1D problem.}\\
We now define the following functions on $[0,+\infty)$ as the supremum/infimum in the tangential variables of the functions defined in $\R\times \overline \Omega$,
\[
  \overline u(\hat x_d):=\sup_{\hat \xi\in \R^d} \left\{U^0(\hat \xi,\hat x_d)-\bar b\cdot \hat \xi\right\},\quad \underline v(\hat y_d):=\inf_{\hat \zeta\in \R^d} \left\{V^0(\hat \zeta,\hat y_d)-\bar b\cdot \hat \zeta\right\}.
\]
From \eqref{eq::n31}, these functions satisfy
\[
  -\infty\le -L_\nu|\hat x_d-\hat y_d|+\overline u(\hat x_d) \le \underline v(\hat y_d)\le +\infty ,\quad 0\le \overline u(0) \le \underline{v}(0)\le 0.
\]
In particular, this implies that $\overline u(0) =0= \underline{v}(0)$.
Because of this one-sided Lipschitz inequality, this is also the case for their semi-continuous envelopes, \textit{i.e.} we have (and this is important)
\begin{equation}\label{eq::n33}
-\infty\le -L_\nu|\hat x_d-\hat y_d|+\overline u^*(\hat x_d) \le \underline v_*(\hat y_d)\le +\infty ,\quad \overline u^*(0) =0= \underline{v}_*(0).
\end{equation}
We set $H_\a(Y,p)=H(Y,p-\a Dg(\bar y))$ and $F_\a(Y,p)=F(Y,p-\a Dg(\bar y))$.
From \eqref{eq::n32}, we get (again from  stability) that these functions satisfy in particular for $\bar X:=(\bar t,\bar x)$, $\bar Y:=(\bar s,\bar x)$
and $\bar b=(\bar b_0,\bar b')\in \R\times \R^{d-1}$
\begin{equation}\label{eq::n34}
\left\{\begin{array}{rlrl}
\bar b_0 + H(\bar X,\bar b',\partial_{\hat x_d} \overline u^*) \le 0 &\quad \text{in}&\quad (0,+\infty)&\cap \left\{|\overline u^*|< +\infty\right\}\\
\bar b_0+ \min(F,H)(\bar X,\bar b',\partial_{\hat x_d} \overline u^*) \le 0 &\quad \text{in}&\quad \left\{0\right\}&\cap \left\{|\overline u^*|< +\infty\right\}\\
\\
-\bar \eta+\bar b_0 + H_\a(\bar Y,\bar b',\partial_{\hat y_d} \underline v_*) \ge 0 &\quad \text{in}&\quad (0,+\infty)&\cap \left\{|\underline v_*|< +\infty\right\}\\
-\bar \eta+\bar b_0 + \max(F_\a,H_\a)(\bar Y,\bar b',\partial_{\hat y_d} \underline v_*) \ge 0 &\quad \text{in}&\quad  \left\{0\right\}&\cap \left\{|\underline v_*|< +\infty\right\}.
\end{array}\right.
\end{equation}
\medskip

\noindent {\bf Step 7: getting a contradiction from structural assumptions.}\\
We now apply Corollary \ref{cor::n6}. In order to do so, we now set $z=\hat x_d=\hat y_d$ and consider
\[
\overline p_d:=\limsup_{[0,+\infty)\ni z\to 0} \frac{\overline u^*(z)}{z},\quad \underline p_d:=\liminf_{[0,+\infty)\ni z\to 0}
\frac{\underline v_*(z)}{z},\quad a_d:=\min(\underline p_d, \overline p_d),\quad b_d:=\max(\underline p_d, \overline p_d)
\]
and we get that there exists $p_d\in [a_d,b_d]\cap \R\not= \emptyset$ such that either
\[
\bar b_0 + H(\bar X,\bar b',p_d)\le 0 < \bar \eta \le H_\a(\bar Y,\bar b',p_d)-H(\bar X,\bar b',p_d)
\]
or
\[
\bar b_0 + \max(F,H)(\bar X,\bar b',p_d)\le 0 < \bar \eta \le F_\a(\bar Y,\bar b',p_d)-F(\bar X,\bar b',p_d).
\]
One of these facts are true along a subsequence $\nu \to 0$. In the first case, we get from the assumption on the Hamiltonian $H$, see \eqref{eq::b4} ii), that (using $\bar p'=\bar b'$) and again
$L := 2 \tilde L_\nu$ for $\alpha$ small enough,
\begin{align*}
\bar \eta \le H_\a(\bar Y,\bar b',p_d)-H(\bar X,\bar b',p_d) & \le \omega\left(|\bar X-\bar Y| \cdot \left[1+|\bar b'|+\max\left\{0,H(\bar X,\bar b',p_d)\right\}\right]\right)+\omega_L(\a Dg(\bar y))\\
&\le \omega\left(|\bar \xi-\bar \zeta|\cdot \left[1+|\bar b'|+\max\left\{0,-\bar b_0\right\}\right]\right)+\omega_L(\a Dg(\bar y))\\
&\le  \omega\left(2\frac{|\bar \xi-\bar \zeta|^2}{\nu} + |\bar \xi-\bar \zeta|\right) +\omega_L(\a Dg(\bar y)) \to  0 \quad \text{as}\quad \a \to 0, \text{ and then } \nu\to 0
\end{align*}
where we have used the expression of $\displaystyle \bar b=\frac{\bar \xi-\bar \zeta}{\nu}$  in the third line, and \eqref{eq::c123xt} in the last line. Contradiction because $\bar \eta\ge \eta/T^2>0$.\\
From the assumption on the function $F$, see \eqref{eq::b5} ii), we  get a similar contradiction in the second case,
\begin{align*}
\bar \eta \le F_\a(\bar Y,\bar b',p_d)-F(\bar X,\bar b',p_d) &\le \omega\left(|\bar \xi-\bar \zeta|\cdot \left[1+|\bar b'|+\max\left\{0,\max\left\{F,H\right\}(\bar X,\bar b',p_d)\right\}\right]\right)+\omega_L(\a Dg(\bar y))\\
&\le \omega\left(|\bar \xi-\bar \zeta| \cdot \left[1+|\bar b'|+\max\left\{0,-\bar b_0\right\}\right]\right)+\omega_L(\a Dg(\bar y))\\
&\le \omega\left(2\frac{|\bar \xi-\bar \zeta|^2}{\nu} + |\bar \xi-\bar \zeta| \right) +\omega_L(\a Dg(\bar y)) \to 0 \quad \text{as} \quad \a \to 0, \text{ and then } \nu\to 0.
\end{align*}
We conclude that $M\le 0$. Recalling that
\[ M = \sup_{t \in [0,T), x \in \overline \Omega} \left\{ \tilde u(t,x) - v(t,x) - \frac{\eta}{T-t} \right\} \le 0, \]
it is enough to let $\eta \to 0$ to get $u \le \tilde u \le v$ as desired.
\medskip

\noindent {\bf Back to Step 3: proof of the key one-sided Lipschitz estimate \eqref{eq::a23}}\\
We now justify \eqref{eq::a23}. Following Lemma \ref{lem::62xt}, we extend $\tilde u$  and consider
$$\tilde U^\nu(\zeta,x_d):=\sup_{\xi\in \R\times \R^{d-1}} \left\{\tilde u(\xi,x_d)-\frac{|\zeta-\xi|^2}{2\nu}\right\}$$
and there exists some (possibly non unique) $\bar \xi_\zeta\in [s-\theta^\nu,s+\theta^\nu]\times \R^{d-1}$ such that $\tilde U^\nu(\zeta,x_d)=\tilde u(\bar \xi_\zeta,x_d)-\frac{|\bar \xi_\zeta-\zeta|^2}{2\nu}$. If $s\in [\theta^\nu,T-\theta^\nu]$, then we see that $\bar \xi_\zeta\in (0,T)\times \R^{d-1}$ and  we also have
$$\tilde U^\nu(\zeta,x_d):=\sup_{\xi\in [0,T)\times \R^{d-1}} \left\{\tilde u(\xi,x_d)-\frac{|\xi-\zeta|^2}{2\nu}\right\}.$$
In particular for $(\zeta,x_d)=(\bar \zeta,\bar x_d)$, we can choose $\bar \xi_{\bar \zeta}=\bar \xi$ where $\bar \xi_{\bar \zeta}$ is given by Lemma \ref{lem::62xt}
and $\bar X=(\bar \xi,\bar x_d)$, $\bar Y=(\bar \zeta,\bar x_d)$ appear in \eqref{eq::d7}.
Now we choose $\nu>0$ small enough such that $\theta^\nu<\tau_\eta$, and we set $I^\nu:=(\theta^\nu,T-\theta^\nu)$. Moreover we have for all $\zeta \in I^\nu\times \R^{d-1}$, $y_d\in [0,+\infty)$, 
$$\tilde U^\nu(\zeta,y_d)-V(\zeta,y_d) \le \sup_{\xi \in [0,T)\times \R^{d-1},\ x_d\in [0,+\infty)}\Psi_{\nu,\alpha}(\xi,\zeta,y_d)= \Psi_{\nu,\alpha}(\bar \xi,\bar \zeta,\bar x_d)=\tilde U^\nu(\bar \zeta,\bar x_d)-V(\bar \zeta,\bar x_d).$$
Now from Lemma \ref{lem::62xt}, we also know that $\tilde U^\nu$ is $L_\nu$-Lipschitz, and then $\tilde U^\nu(\zeta,x_d)-\tilde U^\nu(\zeta,y_d)\le L_\nu |x_d-y_d|$, which implies
$$\tilde U^\nu(\zeta,x_d)-V(\zeta,y_d) \le \tilde U^\nu(\bar \zeta,\bar x_d)-V(\bar \zeta,\bar x_d)+L_\nu|x_d-y_d|$$
which gives exactly \eqref{eq::a23}. This ends the proof of the theorem.

\end{proof}
We now turn to the proof of Lemma \ref{lem::a15}.
\begin{proof}[Proof of Lemma \ref{lem::a15}]
Recall that we have 
$$\tilde u(t,x) \le u_0(x)+C_T,\quad v(t,x) \ge u_0(x)-C_T.$$
Hence
$$\Psi_{\nu,\alpha,\delta}(t,x,s,y) \le 2C_T+B_\delta(x,y)-\alpha g(y)-\frac{\eta}{T-s}-\frac{|t-s|^2}{2\nu}$$
with
$$B_\delta(x,y):=\left\{u_0(x)-u_0(y)-\frac{|x_d-y_d|^2}{2\delta}-\frac{|x'-y'|^2}{2\nu}\right\}\le \phi_\nu(|x'-y'|)+\phi_\delta(|x_d-y_d|)$$
and 
$$\phi_\delta(r):=L_0r-\frac{r^2}{2\delta}\le \frac{\delta L_0^2}{2}.$$
Here $\phi_\delta$ is concave with $\phi_\delta(r_\delta)=0$ for $r_\delta:=2\delta L_0$. Moreover
$$\phi_\delta(r)\le (r-r_\delta)\phi'_\delta(r_\delta)=(r-r_\delta)(L_0-\frac{r}{\delta})$$
i.e.
\begin{equation}\label{eq::c121xt}
\phi_\delta(r)\le -\delta^{-1}(r-r_\delta)^2\quad \mbox{for}\quad r\ge r_\delta=2\delta L_0.
\end{equation}
We get in particular
$$0<\Psi_{\nu,\alpha,\delta}(X,Y)\le 2C_T+\phi_\delta(|x_d-y_d|)+\phi_\nu(|x'-y'|)$$
and then
$$0<\Psi_{\nu,\alpha,\delta}(X,Y)\le 2C_T+\phi_\delta(|x_d-y_d|)+\frac{\nu L_0^2}{2}$$
which implies from \eqref{eq::c121xt} that
\begin{equation}\label{eq::c100x}
|x_d-y_d| \le \sqrt{2\delta \left\{C_T+\frac{\nu L_0^2}{4}\right\}}+2\delta L_0
\end{equation}
and symmetrically that
\begin{equation}\label{eq::c100xb}
|x'-y'| \le \sqrt{2\nu \left\{C_T+\frac{\delta L_0^2}{4}\right\}}+2\nu L_0.
\end{equation}

We also deduce from $0<\Psi_{\nu,\alpha,\delta}(X,Y)$  that
\begin{equation}\label{eq::a16}
\alpha g(y)+\frac{\eta}{T-s}+\frac{|t-s|^2}{2\nu}\le 2C_T+\frac{(\nu+\delta) L_0^2}{2}\le 2C_T+\frac{\eta}{T} \le 3C_T
\end{equation}
for $\eta>0$ small enough (the size of $\eta$ depending  on $C_T$ and $T$, but not on $\nu,\alpha,\delta$), and for $\delta,\nu>0$ small enough (for a size depending on $\eta$).
Therefore we have
\begin{equation}\label{eq::b100}
|t_k-s_k|\le \bar \theta^\nu:=3\sqrt{\nu C_T}
\end{equation}
and
$$T-s_k> \frac{\eta}{3C_T},\quad T-t_k> \frac{\eta}{3C_T}-\bar \theta^\nu\ge \frac{\eta}{4C_T}$$
for $\nu>0$
small enough (for a size depending on $\eta$).\medskip

Similarly, from Lemma \ref{lem::60xt} on the barriers (in particular using Case 1 of the proof, for Lipschitz initial data $u_0$), we know that there exists some $\lambda>0$ such that
$$u(t,x) \le u_0(x)+\lambda t,\quad v(s,y) \ge u_0(y)-\lambda s.$$
Hence
$$\Psi_{\nu,\alpha,\delta}(P_k)\le \lambda(t+s)+L_0|x-y|-\frac{\eta}{T}\le \lambda(t_k+s_k)-\frac{2\eta}{3T}$$
where we have used bound \eqref{eq::c100x}-\eqref{eq::c100xb} for $\delta,\nu>0$ small enough (for a size depending on $\eta>0$).
Therefore 
$$\max(t,s) > \frac{\eta}{3\lambda T},\quad \min(t,s)> \frac{\eta}{3\lambda T}-\bar \theta^\nu\ge \frac{\eta}{4\lambda T}$$
for $\nu>0$ small enough (for a size depending on $\eta$).
Up to increase $\lambda$ or $C_T$ (and decrease $\nu>0$ if necessary), we can assume that
$$\lambda T\equiv C_T.$$
Setting
$$\tau_\eta:=\frac{\eta}{4\lambda T}=\frac{\eta}{4C_T}$$
and for $\nu>0$ small enough, 
we see that
$$X,Y \in [\tau_\eta,T-\tau_\eta]\times \overline \Omega.$$
This gives the result with \eqref{eq::c100x}, \eqref{eq::c100xb} and \eqref{eq::a16}. This ends the proof of the lemma.

\end{proof}

\section{The comparison principle on a bounded domain}
\label{s5}

Let us consider an  open set $\Omega$ satisfying for, $d\ge 1$,
\begin{equation}\label{eq::b0}
\mbox{$\Omega\subset \R^d$ is a bounded open set with $C^1$ boundary and outward unit normal $n(x)$}.
\end{equation}

Let $T>0$.
We consider the following equation for $u(t,x)$ with $X:=(t,x)\in [0,T)\times \overline \Omega$
\begin{equation}\label{eq::b1b}
u_t+ H(X,Du)=0 \quad \mbox{on}\quad (0,T)\times \Omega
\end{equation}
and the boundary condition
$$u_t+ F(X,Du)=0 \quad \mbox{on}\quad (0,T)\times \partial \Omega.$$
We also consider an initial boundary condition
\begin{equation}\label{eq::b3b}
u(0,\cdot)=u_0\quad \mbox{on}\quad \left\{0\right\}\times \overline \Omega.
\end{equation}

The rigorous meaning of desired   boundary conditions is the following,
\begin{equation}\label{eq::b2b}
\left\{\begin{array}{lll}
u_t+ \min(F,H)(X,Du)\le 0 &\quad \mbox{on}\quad (0,T)\times \partial \Omega &\quad \mbox{(for subsolutions)}\\
u_t+ \max(F,H)(X,Du)\ge 0 &\quad \mbox{on}\quad (0,T)\times \partial \Omega &\quad \mbox{(for supersolutions)}\\
\end{array}\right.
\end{equation}

As far as Hamiltonians are concerned, we assume the following structure conditions, where $\omega,\omega_L$ are  moduli of continuity.
\begin{equation}\label{eq::b4b}
\left\{\begin{array}{l}
\mbox{\bf i) (Continuity)}\\
\mbox{$H:[0,T]\times \overline \Omega \times \R^d \to \R$ is continuous}\\
\\
\mbox{\bf ii) (Uniform continuity in the gradient)}\\ 
\mbox{For any $L>0$, we have}\\
|H(X,p)-H(X,q)| \le \omega_L(|p-q|)\quad \mbox{for all}\quad X\in [0,T]\times \overline \Omega,\quad p,q\in [-L,L]^d\\
\\
\mbox{\bf iii) (Quantified continuity in time-space variables)}\\
H(Y,p)-H(X,p) \le \omega(|Y-X|\left(1+\max\left\{0,H(X,p)\right\}\right))\quad \mbox{for all}\quad \left\{\begin{array}{l}
X,Y\in [0,T]\times \overline \Omega\\
p\in \R^d
\end{array}\right.
\\
\mbox{\bf iv) (Uniform  coercivity)}\\
\displaystyle \lim_{|p|\to +\infty} \inf_{X\in [0,T]\times \overline \Omega} H(X,p) = +\infty\\
\end{array}\right.
\end{equation}
and as previously, making artificially appear the dependence on $x\in \overline \Omega$ for $F$ (in order to unify the presentation of $H$ and $F$), we consider
for $X=(t,x)$
\begin{equation}\label{eq::b5b}
\left\{\begin{array}{l}
\mbox{\bf i) (Continuity)}\\
\mbox{$F:[0,T]\times \partial \Omega \times \R^d \to \R$ is continuous}\\
\mbox{and the map $q\mapsto F(X,p-qn(x))$ is nonincreasing}\\
\\
\mbox{\bf ii) (Uniform continuity in the gradient)}\\ 
\mbox{For any $L>0$, we have}\\
|F(X,p)-F(X,q)| \le \omega_L(|p-q|)\quad \mbox{for all}\quad X\in [0,T]\times \partial \Omega,\quad p,q\in [-L,L]^d\\
\\
\mbox{\bf iii) (Continuity in the tangential variables)}\\
F(Y,p)-F(X,p) \le \omega(|Y-X|\left(1+\max\left\{0,\max(F,H)(X,p)\right\}\right))\quad \mbox{for all}\quad \left\{\begin{array}{l}
X,Y\in [0,T]\times \partial \Omega\\
p\in \R^d\\
\end{array}\right.\\
\mbox{\bf iv) (Uniform normal semi-coercivity)}\\
\mbox{For any $L>0$, we have}\\
\displaystyle \lim_{q\to -\infty} \inf_{X\in [0,T]\times \partial \Omega,\ p\in  [-L,L]^{d}} F(X,p-qn(x)) = +\infty.\\
\end{array}\right.
\end{equation}

We then have the following theorem.
\begin{theo}[Comparison principle  on a bounded open set $\Omega$]\label{th::40o}
Let $T>0$ and assume that $H,F$ satisfy respectively \eqref{eq::b4b} and \eqref{eq::b5b}. Assume that the initial data $u_0$ is continuous.
Let $u,v: [0,T)\times \overline \Omega \to \R$ be two functions with $u$ upper semi-continuous and $v$ lower semi-continuous. 
Assume that $u$ (resp. $v$) is a viscosity subsolution (resp. supersolution) of \eqref{eq::b1b}-\eqref{eq::b3b}. Assume moreover that there exists a constant $C_T>0$ such that
$$u\le u_0+C_T\quad \mbox{and}\quad v\ge u_0-C_T\quad \mbox{on}\quad [0,T)\times \overline \Omega.$$
If 
$$u(0,\cdot)\le u_0 \le v(0,\cdot) \quad \mbox{on}\quad \left\{0\right\}\times \overline \Omega,$$
then 
$$u\le v \quad \mbox{in}\quad [0,T)\times \overline \Omega.$$
\end{theo}

In order to give the proof of Theorem \ref{th::40o}, we need the following lemma which proof is left to the reader.
\begin{lem}[Action of a diffeomorphism on the structural conditions satisfied by $H,F$]\label{lem::b45}
Assume that $\Omega$ has the regularity given in assumption \eqref{eq::b0}. For $T>0$, we set $Q_T:=(0,T)\times \Omega$ with $\Omega\subset \R^d$ and for $P=(p_0,p)\in \R\times \R^d$, we set
$$H_0(Y,P):=p_0+H(Y,p),\quad F_0(Y,P):=p_0+F(Y,p).$$ We say by extension that $H_0,F_0$ satisfy respectively \eqref{eq::b4b} and \eqref{eq::b5b}, if  $H,F$ do it.

Assume that $H_0,F_0$ satisfy respectively \eqref{eq::b4b} and \eqref{eq::b5b} and for $x_0\in \partial \Omega$,   
consider, locally around $x_0$, a $C^1$-diffeomorphism $\Phi$ from    $\overline \Omega$ to $\overline{\tilde \Omega}$ with $\Phi(x_0)=y_0\in \partial \tilde \Omega$, that we extend by the identity on the time variable. Still denoting by $\Phi$ this diffeomorphism, we assume that $\Phi$ maps locally $\overline{Q}_T$ to $\overline{\tilde Q}_T \subset \R^{1+d}$ with locally $\Phi(\partial Q_T)=\partial \tilde Q_T$. For $Y\in \tilde Q_T$ and ${P}\in \R^{1+d}$, we set
$$\left\{\begin{array}{ll}
\tilde{H}_0(Y,{P}):=H_0(\Phi^{-1}(Y), {P}\cdot B(Y)) &\quad \mbox{locally around $[0,T]\times \left\{y_0\right\}$, on}\quad \tilde Q_T\\
\tilde{F}_0(Y,{P}):=F_0(\Phi^{-1}(Y), {P} \cdot B(Y)) &\quad \mbox{locally around $[0,T]\times \left\{y_0\right\}$, on}\quad \partial \tilde Q_T\\
\end{array}\right.$$
with
$$(P\cdot B)_j=\sum_{i=0}^d P_i \ \left\{(D_j\Phi_i)\circ \Phi^{-1}\right\}\quad \mbox{for}\quad j=0,\dots,d.$$
Then $\tilde H_0,\tilde F_0$ satisfy respectively \eqref{eq::b4b} and \eqref{eq::b5b} locally around $[0,T]\times \left\{y_0\right\}$, with some suitable moduli.
\end{lem}
We now turn to the proof of Theorem \ref{th::40o}.
\begin{proof}[Proof of Theorem \ref{th::40o}]
Up to proceed as in Step 0 of the proof of Theorem \ref{th::b7}, we can assume that $u_0$ belongs to $C^1(\overline \Omega)$.\newline
We set
$$\tilde u := \max\left\{u,\underline u_0\right\},\quad \underline u_0(t,X)=u_0(X)-C_T=\underline u_0(X)$$
and
$$M:=\sup_{(t,x)\in [0,T)\times \overline \Omega} \Psi(t,x)\quad \mbox{with}\quad \Psi(t,X)=\tilde u(t,x)-v(t,x)-\frac{\eta}{T-t}.$$
Assume by contradiction that 
$$0<M=\Psi(X_0)\quad \mbox{with}\quad X_0:=(t_0,x_0)\in [0,T)\times \overline \Omega.$$
By assumption, we have $t_0>0$. If $x_0\in \Omega$, then we can localize, and then get a contradiction by standard method of doubling of variables.
Hence assume that 
$x_0\in \partial \Omega$.
Up to modify slithly the functions, we can assume that the suppremum is strict at $X_0$.
Up to change the coordinates, we can also assume that 
$$x_0=0,\quad \Omega\cap B_r(x_0)=\left\{x=(x',x_d)\in \R^{d-1}\times \R,\quad x_d>h(x')\right\}\cap B_r(x_0),\quad h(0)=0=Dh(0).$$
Setting
$$x=\Phi(y)\quad \mbox{with}\quad y=(y',y_d)=\Phi^{-1}(x):=(x',x_d-h(x'))\quad \mbox{and}\quad \tilde U(t,y):=\tilde u(t,\Phi(y)),\quad V(t,y):=v(t,\Phi(y))$$
we see that $\Phi$ is locally invertible and its inverse is a $C^1$ map, given, for some $\rho>0$ small enough, by
$$\Phi: K_\rho^+\to \overline \Omega \cap B_r(x_0)\quad \mbox{with}\quad K_\rho^+:=[-\rho,\rho]^{d-1}\times [0,\rho].$$
Hence we have
$$x=(x',x_d)=(y',y_d+h(y'))=\Phi(y),\quad \tilde u(t,x)=\tilde U(t,\Phi^{-1}(x)),\quad v(t,x)=V(t,\Phi^{-1}(x)),$$
and
$$\left\{\begin{array}{ll}
D_{x_d} \tilde u(t,x)&=D_{y_d} \tilde U(t,\Phi^{-1}(x)),\\
D_{x'}\tilde u(t,x)&=D_{y'}\tilde U(t,\Phi^{-1}(x))-\left\{D_{y_d}\tilde U(t,\Phi^{-1}(x))\right\}\cdot D_{x'}h(x'),\\
\tilde u_t (t,x)&=\tilde U_t(t,\Phi^{-1}(x)).
\end{array}\right.$$
This gives the new Hamiltonian $\tilde H$ and boundary function $\tilde F$ for $Y=(t,y)$, $X=(t,x)$ and $x=\Phi(y)$
$$\begin{array}{l}
\tilde H(Y,D\tilde U(Y))=H(X,D\tilde u(X)),\\
\tilde F(Y,D\tilde U(Y))=F(X,D\tilde u(X)),\quad \mbox{for}\quad y=(y',0)
\end{array}$$
which are defined by (for $y:=(y',y_d)$)
$$\begin{array}{l}
\tilde H(t,y,p',p_d):=H(t,\Phi(y),p'-p_d D_{y'}h(y'),p_d),\\
\tilde F(t,y,p',p_d):=F(t,\Phi(y),p'-p_dD_{y'}h(y'),p_d),\quad \mbox{for}\quad y=(y',0).
\end{array}$$
Hence $\tilde U$ and $V$ are respectively sub/supersolutions of
$$\left\{\begin{array}{l}
W_t+\tilde H(Y,DW)=0\quad \mbox{on}\quad (0,T)\times [-\rho,\rho]^{d-1}\times [0,\rho]\\
W_t+\tilde F(Y,DW)=0\quad \mbox{on}\quad (0,T)\times [-\rho,\rho]^{d-1} \times \left\{0\right\}\\
\end{array}\right.$$
We now apply Lemma \ref{lem::b45} to insure that $\tilde H$ and $\tilde F$ satisfy (locally) the same structural conditions than $H$ and $F$.
Moreover, we have 
$$M=\sup_{(t,y)\in [0,T)\times K_\rho^+} \tilde \Psi(t,y)=\tilde \Psi(t_0,0)\quad \mbox{with}\quad \tilde \Psi(t,y):=\tilde U(t,y)-V(t,y)-\frac{\eta}{T-t}.$$
Up to add some small and smooth tangential correction term $|t-t_0|^2+|y'|^2$ to $V$ (here we neglect this correction which can be treated in a very classical way), we can assume that
$$\tilde \Psi(t,y)< M \quad \mbox{for all}\quad (t,y)\in \left([0,T)\times K_\rho^+\right) \backslash \left\{(t_0,0)\right\}.$$
This implies that for $\xi=(t,x')$, $\zeta=(s,y')$
$$M(\theta):=\sup \left\{\tilde U(\xi,y_d)- V(\zeta,y_d) -\frac{\eta}{T-s},\quad \xi, \zeta\in [0,T)\times [-\rho,\rho]^{d-1},\ y_d\in [0,\rho],\quad |\xi-\zeta|\le \theta\right\}$$
with
$$\lim_{\theta \to 0^+} M(\theta)=M>0.$$
We are then back to the begining of the proof of  Theorem~\ref{th::b7}, which leads to a contradiction. Again, we conclude that $M\le 0$ for all $\eta \to 0^+$, and then deduce that $\tilde U\le V$, and then $u\le v$. This ends the proof of the theorem.
\end{proof}


\end{document}